\documentclass[11pt]{article}
\usepackage{amsmath}
\usepackage{subfigure}% Support for small, `sub' figures and tables
\usepackage{amsfonts}
\usepackage{float}
\usepackage{amsthm}
\usepackage{amssymb, setspace}
\usepackage{color,graphicx, epsfig, geometry, hyperref,fancyhdr}
\usepackage{mathrsfs}
\usepackage[T1]{fontenc}
\usepackage{latexsym,amssymb,amsmath,amsfonts,amsthm}
\usepackage{graphicx}
\usepackage{color}
\usepackage{soul}
\usepackage{caption}
\usepackage{dsfont}
\usepackage{multirow}
\usepackage{rotating}
\newtheorem{theorem}{Theorem}[section]

\newtheorem{lemma}{Lemma}[section]

\newtheorem{assumption}{Assumption}[section]

\newcommand{\N}{\mathbb{N}}

\newcommand{\R}{\mathbb{R}}
\newcommand{\C}{\mathbb{C}}

\newcommand{\grad}{\nabla}

\graphicspath{{pics/}}

\begin{document}
\setlength{\parskip}{1mm}
\setlength{\oddsidemargin}{0.1in}
\setlength{\evensidemargin}{0.1in}
\lhead{}
\rhead{}

\begin{center}
{\bf \Large \noindent The inverse scattering problem for a conductive boundary condition and transmission eigenvalues }\\
\vspace{0.2in}
Isaac Harris\\
Department of Mathematics\\
Texas A $\&$ M University \\
College Station, Texas USA 77843-3368\\
E-mail: iharris@math.tamu.edu\\

\vspace{0.2in}

Andreas Kleefeld \\
Forschungszentrum J\"{u}lich GmbH \\
Institute for Advanced Simulation\\
J\"{u}lich Supercomputing Centre \\
Wilhelm-Johnen-Stra{\ss}e, 52425 J\"{u}lich, Germany\\
Email: a.kleefeld@fz-juelich.de\\
\end{center}

\begin{abstract}
\noindent In this paper, we consider the inverse scattering problem associated with an inhomogeneous media with a conductive boundary. First, we discuss the inverse conductivity problem of reconstructing the conductivity parameter from scattering data. 
Next, we consider the corresponding interior transmission eigenvalue problem. This is a new class of eigenvalue problem that is not elliptic, not self-adjoint, and non-linear, which gives the possibility of complex eigenvalues. We investigate the convergence of the eigenvalues as the conductivity parameter tends to zero as well as prove existence and discreteness for the case of an absorbing media. Lastly, several numerical and analytical results support the theory and we show that the inside-outside duality method can be used to reconstruct the interior conductive eigenvalues.
\end{abstract}

\noindent {\bf Keywords}: inverse scattering, conductive boundary condition, transmission eigenvalues.

%%%%%%%%%%%%%%%%%%%%%%%%%%%%%%%%%%%%%%%%%%%%%%%%%%%%%%%%%%%%%%%%%
\section{Introduction and problem statement}

In this section, we formulate the problems under consideration. We are interested in recovering a conductive boundary coefficient from measured scattering data and/or the so-called interior transmission eigenvalues. These eigenvalue problems can be derived by studying the injectivity of the operators used in qualitative inversion methods. The interior transmission eigenvalue problem is non-linear and non self-adjoint which makes their investigation interesting from the mathematical point of view, but also a challenging task from the computational point of view.
The inverse problem we consider here corresponds to the following direct scattering problem: find the total field $u \in H^1(D)$ and scattered field $u^s \in H^1_{loc}(\R^d \setminus \overline{D})$ (where $d=2,3$) such that  
\begin{align}
\Delta u^s +k^2 u^s=0  \quad \textrm{ in }  \R^d \setminus \overline{D}  \quad  \text{and} \quad  \Delta u +k^2 nu=0  \quad &\textrm{ in } \,  {D}  \label{direct1}\\
 (u^s+u^i )^+ - u^-=0  \quad  \text{and} \quad \partial_\nu (u^s+u^i )^+ + \eta (u^s+u^i )^+= {\partial_{\nu} u^-} \quad &\textrm{ on } \,  \partial D \label{direct3}\\
 \lim\limits_{r \rightarrow \infty} r^{{(d-1)}/{2}} \left( {\partial_r u^s} -\text{i} k u^s \right)=0 \label{SRC}&.
\end{align}
The superscripts $+$ and $-$ indicate the trace on the boundary taken from the exterior or interior of the domain $D$, respectively. The total field is given by $u=u^s+u^i$ where the incident field is either given by a plane wave $u^i=\text{e}^{\text{i} k x \cdot \hat{y} }$ or a point source $u^i = \Phi_k (x,y)$ where
$$\Phi_k (x,y)= \left\{\begin{array}{lr} \frac{ \text{i} }{4} H^{(1)}_0 (k | x-y |) \, \, & \, \text{in} \, \, \R^2 \\
 				&  \\
 \frac{1}{4 \pi} \frac{\text{e}^{ \text{i} k | x-y |} }{| x-y |}  & \,  \text{in} \,\,   \R^3.
 \end{array} \right. $$  
 Here the incident direction $\hat{y}$ given by a point on the unit sphere/circle denoted $\mathbb{S}$ for a plane wave where as the source point $y$ is some point on a $\mathcal{C}^2$ curve/surface denoted $C$. We let $D \subset \R^d$ be a bounded simply connected open set with $\nu$ being the unit outward normal to the boundary where the boundary $\partial D$ is of class $\mathcal{C}^2$. Here we assume that $\overline{D} \subset \text{int}(C)$. Furthermore, we assume that the refractive index $n(x)\in L^{\infty}(D)$ and the conductivity parameter $\eta (x) \in L^{\infty}(\partial D)$. Note that the refractive index and conductivity parameters are allowed to take complex values which physically represents an absorbing media.

The interior transmission eigenvalue problem with conductive boundary conditions associated with \eqref{direct1}--\eqref{SRC} reads: find $k \in \C$ and nontrivial $(w,v) \in L^2(D) \times L^2(D)$ where $w-v \in  H^2(D) \cap H^1_0(D)$ such that  
\begin{align}
\Delta w +k^2 n w=0 \quad \text{and} \quad \Delta v + k^2 v=0  \quad &\textrm{ in } \,  D \label{teprob1} \\
 w-v=0  \quad  \text{and} \quad  {\partial_{\nu} w}-{\partial_\nu v}= \eta v \quad &\textrm{ on } \partial D \label{teprob2} 
\end{align} 
where we define the Sobolev space 
$$H^2(D) \cap H^1_0(D) = \big\{u\in H^2(D) \, : \, u=0 \, \text{ on } \partial D \big\}$$
equipped with the $H^2(D)$ norm. This eigenvalue problem can be derived from considering an obstacle that `does not' scatter for a specific wave number $k$ and incident field. The interior eigenvalue problem for an inhomogeneous media with a conductive boundary was studied in \cite{te-cbc}, where they used the theoretical results in \cite{chtevexist} to prove existence of real eigenvalues. 
Their numerical experiments seem to indicate that the interior conductive eigenvalues can be reconstructed using the 
linear sampling method and that they converge to the interior transmission eigenvalues as the conductivity parameter goes to zero. Here we  prove the convergence of the  eigenvalues as the conductivity parameter goes to zero as well as prove the existence and discreteness of the transmission eigenvalues for an absorbing media.

Notice that from the conductive eigenvalue problem \eqref{teprob1}--\eqref{teprob2} we have that the difference $u=w-v \in H^2(D) \cap H^1_0(D)$ satisfies 
$$ \Delta u +k^2 n u=-k^2(n-1)v  \quad \textrm{ in } \, D, \quad u=0 \, \, \textrm{ and } \, \, {\partial_\nu}u =  \eta v \, \textrm{ on } \, \partial  D.$$
This implies that the conductive eigenvalue problem can be written as: 
find the values $k \in \C$ such that there is a nontrivial solution $u \in H^2(D) \cap H^1_0(D)$ satisfying
\begin{align}
(\Delta+k^2)\frac{1}{n-1}(\Delta u +k^2 n u)=0  \quad &\textrm{ in } \,  D, \label{teprobu3} \\
-\frac{k^2}{\eta} \frac{\partial u}{\partial \nu}=\frac{1}{n-1}(\Delta u +k^2 n u) \quad &\textrm{ on } \partial D.  \label{teprobu4} 
\end{align}
Notice that $v$ and $w$ are related to the eigenfunction $u$ by
$$v = -\frac{1}{k^2(n-1)}(\Delta u + k^2 n u) \quad \text{ and } \quad w = -\frac{1}{k^2(n-1)}(\Delta u + k^2 u). $$
It is clear that the conductive eigenvalue problem \eqref{teprob1}--\eqref{teprob2} is equivalent to the eigenvalue problem \eqref{teprobu3}--\eqref{teprobu4}.

The structure of the paper is the following: in Section 2 it is explained how to recover the material parameter $\eta$ from either near field or far field data emphasizing the first case. Additionally, uniqueness is shown and an algorithm is presented to obtain this conductivity parameter $\eta$ assuming a constant index of refraction. In Section 3, the conductive eigenvalue problem is studied in detail. First, it is shown how to obtain conductive interior transmission eigenvalues from either near field or far field data for non-absorbing media. Next, we show the convergence of the eigenvalues $\eta$ tending to zero as well as the convergence of the corresponding eigenfunctions. Then, existence and discreteness of conductive interior transmission eigenvalues is proven for an absorbing media. 
Numerical experiments complement the theory and are given in Section 4. It is shown that the conductive interior transmission eigenvalues and eigenfunctions converge linearly to the classic interior transmission eigenvalues and eigenfunctions both for the real-valued and complex-valued case. It is also demonstrated that one is able to calculate conductive interior transmission eigenvalues with the inside-outside duality method although the theory still has to be carried out. Finally, a short summary and a conclusion is given in Section 5.

%%%%%%%%%%%%%%%%%%%%%%%%%%%%%%%%%%%%%%
\section{Recovery  of the conductivity from scattering data}
Under physical sign assumptions for the coefficients we have that the forward problem \eqref{direct1}--\eqref{SRC} is well-posed (see \cite{fmconductbc}) and we consider the problem of determining the boundary conductivity from either near field (NF) or far field (FF) data sets given by 
$$ \text{NF-data}=\big\{ u^s(x,y) \,  :  \, \forall \, x,y \in C \big\}  \quad \text{ and } \quad  \text{FF-data}=  \big\{ u^\infty (\hat{x}, \hat{y}) \,  :  \, \forall \,  \hat{x},\hat{y} \in \mathbb{S} \big\} . $$
For the NF-data set we have that the scattered field in the exterior of $D$ is given by Green's representation formula (see \cite{p1})
\begin{eqnarray}
u^s({x}, {y} )= \int\limits_{\partial D} u^s(z,{y}) {\partial_{\nu(z)} } \Phi_k (z,x) -\Phi_k (z,x) {\partial_{\nu(z)} }  u^s(z, {y} ) \, \text{d}s(z) \label{nfpdef}
\end{eqnarray}
for any $x \in \R^d \setminus \overline{D}$ and $y \in C$.
In the case of FF-data the scattered field has the expansion (see \cite{p1})
$$ u^s(x,\hat{y})=\gamma \frac{\text{e}^{\text{i}k|x|}}{|x|^{{(d-1)}/{2}}} \left\{u^{\infty}(\hat{x}, \hat{y} ) + \mathcal{O} \left( \frac{1}{|x|}\right) \right\}\; \textrm{  as  } \;  |x| \to \infty$$
where 
$$\gamma = \frac{ \text{e}^{\text{i}\pi/4} }{ \sqrt{8 \pi k} } \,\,\, \text{in} \,\,\, \R^2 \quad \text{and} \quad  \gamma = \frac{1}{ 4\pi } \,\,\, \text{in} \,\,\, \R^3. $$
The corresponding far field pattern $u^\infty$ can be derived from Green's representation and is given by 
\begin{eqnarray}
u^\infty(\hat{x}, \hat{y} )= \int\limits_{\partial \Omega} u^s(z,\hat{y}) {\partial_{\nu(z)} } \text{e}^{-\text{i}k z \cdot \hat{x} } -\text{e}^{-\text{i}k z \cdot \hat{x}} {\partial_{\nu(z)} }  u^s(z, \hat{y} ) \, \text{d}s(z) \label{ffpdef}
\end{eqnarray}
where we have used the fact that 
$$ \phi_z (\hat{x} ) = \text{e}^{-\text{i}k z \cdot \hat{x} }$$
is the far field pattern for the fundamental solution $\Phi_k (z,x)$. The region $\Omega$ is any subset of $\R^d$ such that $D \subseteq \Omega$. For the case of the NF-data set, we define the well-known near field operator as
\begin{align}
N: L^2(C) \longmapsto L^2(C) \quad \text{ such that } \quad (Ng)(x)=\int_{C} u^s( x, y ) g(y) \text{d}s(y). \label{NF-operator}  
\end{align}
 
Where as in the case of the FF-data set we can define the far field operator by 
\begin{align}
F: L^2(\mathbb{S}) \longmapsto L^2(\mathbb{S}) \quad \text{ such that } \quad (Fg)({\hat x})=\int_{{\mathbb S}} u^\infty(\hat{x}, \hat{y} )  g(\hat{y})\text{d}s({{\hat y}}). \label{FF-operator}  
\end{align}

In this section, we consider the inverse problem of determining the conductivity parameter $\eta$ from the scattering data. The problem of reconstruction $\eta$ from FF-data has been studied in \cite{ccm-findbc} for a purely imaginary conductivity. Therefore, we consider that case of the NF-data set with a general complex valued conductivity. Here it is assumed that the refractive index $n$ is given such that 
$$\text{Re} \big(  n(x) \big) \geq n_{min}>0  \quad \text{and}  \quad \text{Im} \big( n(x) \big) \geq 0$$
holds for almost every $x\in D$ and the conductivity $\eta$ satisfies    
$$\text{Re} \big(  \eta (x) \big) \geq 0  \quad \text{and}  \quad \text{Im} \big( \eta (x) \big) \geq 0$$
for almost every $x\in \partial D$. The analysis presented for the NF-data set can easily be modified for the FF-data set. We then derive a reconstruction algorithm for recovering the conductivity $\eta$ from the measured NF data.

%%%%%%%%%%%%%%%%%%%%%%%%%%%%%%%%%%%%%%
\subsection{Uniqueness for reconstructing the conductivity}
Here we prove uniqueness and stability of our inverse problem. We will assume that the obstacle $D$ is known and we will show that knowledge of the measured NF-data gives uniqueness and stability. In  \cite{ccm-findbc} it is shown that the FF-data set can uniquely determine a purely imaginary conductivity and in \cite{ccm-findbc} a numerical algorithm is derived to determine the conductivity from the FF-data. To avoid repetition we will focus on the case of the NF-data set. We turn our attention to proving uniqueness and stability. To do so, we need the following assumption.
\begin{assumption}
The wave number $k \in \R$ is not a Dirichlet eigenvalue for the operator $\Delta +k^2 n$ in $D$. 
\end{assumption} 
Notice that if $\text{Im}(n) \neq 0$ in $D$ then we have that there are no real Dirichlet eigenvalues for $\Delta +k^2 n$ in $D$. With this assumption we now prove the following lemma.

\begin{lemma}\label{perp-of-total-field}
The set 
$$\mathcal{N}= \big\{ u( \cdot \, ,\, y) \,\,   : \, \,  u \, \text{ is the total field for } \eqref{direct1}-\eqref{SRC} \, \,   \text{for a.e. } \, \, \, y \in C \big\} \subset L^2(\partial D)$$
satisfies that $\mathcal{N}^{\perp}=\{0\}$. 
\end{lemma}
\begin{proof}
To prove the claim assume that $\varphi \in L^2(\partial D)$ is such that 
$$ \int\limits_{\partial D} \varphi  u( \cdot \, ,\, y) \, \text{d}s=0 \quad \text{ for a.e. }\, \, y \in C.$$
We now let $v \in H^1_{loc} (\R^d)$ be the unique radiating solution to  
$$\Delta v +k^2 v=0  \quad \textrm{ in } \,  \, \R^d \setminus \overline{D}  \quad  \text{and} \quad  \Delta v +k^2 nv=0  \quad \textrm{ in } \, \, {D} $$ 
with boundary conditions 
$$ v^+ = v^-   \quad  \text{ and } \quad \partial_\nu v^+ -{\partial_{\nu} v^-}  + \eta v^+= \varphi \quad \textrm{ on } \,\,  \partial D.$$
Since both $u$ and $v$ satisfy the same Helmholtz equation inside $D$ Green's second identity gives that
\begin{align*}
0&=  \int\limits_{\partial D} u \partial_\nu v^- - v \partial_\nu u^- \, \text{d}s \\
&= \int\limits_{\partial D} u [ \partial_\nu v^+ +\eta v -\varphi] - v[ \partial_\nu u^+ +\eta u] \, \text{d}s\\
&= \int\limits_{\partial D} u \partial_\nu v^+ - v \partial_\nu u^+ \, ds - \int\limits_{\partial D} \varphi  u( \cdot \, ,\, y) \, \text{d}s.
\end{align*}
Now, since $\varphi \in \mathcal{N}^{\perp}$ and using the identity $u=u^s+u^i$ we have that 
$$\int\limits_{\partial D} v \partial_\nu u^i - u^i \partial_\nu v \, \text{d}s= \int\limits_{\partial D} u^s \partial_\nu v - v \partial_\nu u^s \, \text{d}s$$
where the traces in these boundary integrals are all exterior traces. Let $B_R$ be the ball centered at the origin with radius $R$ large enough so $\overline{D} \subset B_R$. Similarly by appealing to Green's second identity for $u^s$ and $v$ in $B_R \setminus \overline{D}$ we have that 
$$\int\limits_{\partial D} v \partial_\nu u^i - u^i \partial_\nu v \, \text{d}s= - \int\limits_{\partial B_R} u^s \partial_\nu v - v \partial_\nu u^s \, \text{d}s.$$
By the radiation condition we have that the boundary integral over $\partial B_R$  tends to zero as $R$ tends to infinity. Therefore, since $u^i= \Phi(\cdot\, , \,  y)$ we can conclude that 
$$ 0= \int\limits_{\partial D} v \partial_\nu \Phi_k (\cdot\, , \,  y) - \Phi_k (\cdot\, , \,  y) \partial_\nu v \, \text{d}s = v(y)  \quad \text{ for a.e. }\, \, y \in C.$$
This implies that $v=0$ in the exterior of $D$ and the continuity across the boundary $\partial D$ of the Dirichlet data implies that 
$$ \Delta v +k^2 n v= 0 \quad \text{ in } \, \, D \quad \text{and } \quad v=0 \quad \text{on } \, \, \partial D.$$ 
This gives that $v=0$ in the interior of $D$ since $k^2$ is not an eigenvalue which implies that $\varphi=0$ on $\partial D$. 
\end{proof}

With the previous result we are ready to prove the uniqueness for determining $\eta$ from the NF-data set. 
\begin{theorem}\label{uniqueness}
The set of NF measurements 
$$\big\{ u^s(x,y) \,  \, : u^s  \, \text{ is the scattered field for } \eqref{direct1}-\eqref{SRC} \, \, \, \text{ for a.e. } \,\, \,  x,y \in C \big\} $$
uniquely determines the conductivity $\eta \in L^\infty (\partial D)$ 
\end{theorem}
\begin{proof}
To prove the claim assume that there are two conductivity parameters $\eta_j$ for $j=1,2$ such that the set of NF measurements coincide. This implies that the corresponding total fields $u_j$ must coincide in the exterior of $D$. By the continuity of the trace across $\partial D$ and since $k^2$ is not an eigenvalue we then conclude that $u=u_1 = u_2$ in $\R^d$. The conductivity condition on $\partial D$  
$$ \partial_\nu u^+ -{\partial_{\nu} u^-}  + \eta_1 u^+ = 0 \quad \text{and } \quad  \partial_\nu u^+ -{\partial_{\nu} u^-}  + \eta_2 u^+= 0$$
on $\partial D$ gives that $(\eta_1 - \eta_2) u( \cdot \, , \, y)=0$ for a.e. $y\in C$. So we have that $(\eta_1 - \eta_2) \in \mathcal{N}^{\perp}$ and by Lemma \ref{perp-of-total-field} we can conclude that $\eta_1= \eta_2$ on $\partial D$. 
\end{proof}

%%%%%%%%%%%%%%%%%%%%%%%%%%%%%%%%%%%%%%%%
\subsection{Reconstruction algorithm for a constant refractive index }

We now derive a reconstruction method for the conductivity parameter $\eta$ from the NF-data. To this end, recall that the conductive boundary condition is given by 
$$\partial_\nu u^+ -{\partial_{\nu} u^-}  + \eta(x)  u^+= 0 \quad \text{ on } \, \, \partial D.$$
Therefore, we have that if $n(x)$ for $x \in D$ is known then we can determine the exterior and interior Cauchy data for the total field $u$ on the boundary of $\partial D$.  For the case of NF-data we can assume that the scattered field 
$$u^s(x ,y) = \int_{\partial D} \Phi_k (x ,z) f(z) \,  \text{d}s(z) \quad \text{ where } \, \, f=f( \cdot \, ; y) \in L^2(\partial D)$$
for all $y \in C$ and $x \in \R^d \setminus \overline{D}$. Here $\Phi_k (x ,z)$ is yet again the fundamental solution to the Helmholtz equation with wave number $k$. Therefore, we have that for any $x \in C$ 
\begin{align}
u^s( \cdot \, ,y)= Af \quad \text{ where } \quad (Af)(x)= \int_{\partial D} \Phi_k (x ,z) f(z) \,  \text{d}s(z) \Big|_{C} : L^2(\partial D) \longmapsto L^2(C). \label{scat-recon}
\end{align}
Notice, that since the kernel of the integral operator is analytic we have that $A$ is compact and it is clear that if $k^2$ is not a Dirichlet eigenvalue of $-\Delta$ in $D$ then we also have that $A$ is injective. Since the dual operator 
$$ (A^\top g)(z) = \int_{C} \Phi_k (x ,z) g(x) \,  \text{d}s(x) \Big|_{\partial D} : L^2(C) \longmapsto L^2(\partial D)$$ 
is injective due to the uniqueness of the exterior Dirichlet problem to Helmholtz equation. This gives that the operator $A$ has a dense range. This gives that the solution to \eqref{scat-recon} $f \in L^2(\partial D)$ can be determined by a regularization technique. 
 
Now that the scattered field has been determined by solving \eqref{scat-recon} we can conclude that the total field $u$ is known in the exterior of scatterer $D$ and is given by 
$$ u =  \int_{\partial D} \Phi_k (x ,z) f(z) \,  \text{d}s(z) + \Phi_k (x ,y) \quad \text{ for all } x \in \R^d \setminus \overline{D} \quad \text{and} \quad y \in C.$$
Therefore, using the jump relations we have that the exterior Neumann data is given by 
$$ \partial_{\nu} u^+=  \int_{\partial D} \partial_{\nu (x)}  \Phi_k (x ,z) f(z) \,  \text{d}s(z) - \frac{1}{2} f(x)+ \partial_{\nu (x)} \Phi_k (x ,y) \quad \text{ for all } x \in \partial  D \quad  y \in C.$$

We only need to determine the interior Neumann data for the total field to recover the conductivity. To this end, we can use the interior Dirichlet to Neumann mapping 
$$ T : H^{1/2}(\partial D) \longmapsto H^{-1/2}(\partial D) \quad \text{such that} \quad Tf = \partial_\nu w \big|_{\partial D}$$
where 
$$  \Delta w +k^2 n w= 0 \quad \text{ in } \, \, D \quad \text{and } \quad w=f \quad \text{on } \, \, \partial D.$$
Then, we have that $\partial_\nu u^- = Tu^+$ and $T$ can be written using boundary integral operators. Indeed, for simplicity assume that $n$=constant. We define the boundary integral operators 
$$ (V \varphi)(x) =  \int_{\partial D} \Phi_{k\sqrt{n}} (x ,z) \varphi(z) \,  \text{d}s(z) : H^{-1/2}(\partial D) \longmapsto H^{1/2}(\partial D)$$ 
and 
$$ (K \psi)(x) = \int_{\partial D} \partial_{\nu(z)} \Phi_{k\sqrt{n}} (x ,z) \psi(z) \,  \text{d}s(z) : H^{1/2}(\partial D) \longmapsto H^{1/2}(\partial D)$$
for $x \in \partial D$. Therefore, by appealing to Green's Representation in $D$ and the jump relations we have that 
$$T  f =V^{-1} \left( \frac{1}{2}I +K \right) f  \quad \text{ for all } \, \, f \in  H^{1/2}(\partial D)$$
provided that $k^2n$ is not a Dirichlet eigenvalue of $-\Delta$ in $D$. 
One can consider two methods for solving for the conductivity {for}  $j=1, \cdots ,N$ {with}  $x_j \in \partial D$
$$\text{Direct Solution:  } \quad   \eta(x_j)   =    \frac{  (Tu^+)(x_j)  -   \partial_{\nu} u^+ (x_j)  }{ u^+ (x_j) }$$
and 
$$\text{Least Squares Solution: } \quad  \min\limits_{\eta (x)  \in L^{\infty}(\partial \Omega)} \sum\limits_{j=1}^N \Big| \partial_{\nu} u^+ (x_j)  - (Tu^+)(x_j) + \eta(x_j){ u^+ (x_j) }  \Big|^2.$$

%%%%%%%%%%%%%%%%%%%%%%%%%%%%%%%%%%%%%%
\section{The conductive eigenvalue problem}

In this section, we consider the transmission eigenvalue problem for a material with a conductive boundary. We will consider the case of an absorbing material and a non-absorbing material. In the case of an absorbing scatterer we prove the existence and discreteness of the eigenvalues. For the non-absorbing case we focus on studying the behavior of the eigenvalues as the conductivity tends to zero.   
To study this eigenvalue problem we consider the variational problem associated with \eqref{teprobu3}--\eqref{teprobu4}. Therefore, taking a test function $\varphi \in H^2(D) \cap H^1_0(D)$ yields 
\begin{align*}
\hspace*{-1cm} 0&= \int\limits_D  \overline{\varphi} (\Delta+k^2)\frac{1}{n-1}(\Delta u +k^2 n u) \, \text{d}x \\
   \iff 0&=  \int\limits_D   \overline{\varphi}   \Delta \left( \frac{1}{n-1}(\Delta u +k^2n u)  \right) \, \text{d}x + k^2 \int\limits_D \frac{1}{n-1} \overline{\varphi} (\Delta u +k^2 nu)  \, \text{d}x. 
\end{align*}
By Green's second theorem we have that 
\begin{align*}      
\hspace*{-1cm} 0&=  \int\limits_D  \Delta \overline{\varphi} \left( \frac{1}{n-1}(\Delta u +k^2 nu)  \right) \, \text{d}x + k^2 \int\limits_D \frac{1}{n-1}  \overline{\varphi} (\Delta u +k^2 n u)  \, \text{d}x\\
   &\hspace{0.5in} -  \int\limits_{\partial D} \frac{\partial \overline{\varphi} }{\partial \nu}  \left( \frac{1}{n-1}(\Delta u +k^2 n u)  \right) \, \text{d}s.
\end{align*}
Applying the conductivity condition we have that    
 \begin{align}        
0 =  \int\limits_D \frac{1}{ n-1 }(\Delta u +k^2 nu) (\Delta \overline{\varphi} +k^2 \overline{\varphi}) \, \text{d}x +  \int\limits_{\partial D} \frac{k^2}{\eta} \frac{\partial u}{\partial \nu} \frac{\partial \overline{\varphi} }{\partial \nu} \, \text{d}s. \label{TE-varform} 
\end{align}

In \cite{te-cbc} the variational form given in \eqref{TE-varform} was studied for the case when $n$ and $\eta$ are real-valued. In most applications both $n$ and $\eta$ can be complex-valued. So we begin by considering the case of an absorbing material with complex valued conductivity. It is shown numerically that the transmission eigenvalues can be determined from the FF-data in \cite{te-cbc} for a non-absorbing material and following the analysis in \cite{cchlsm} one can prove the following results using either the far field or near field operator, defined in \eqref{FF-operator} and \eqref{NF-operator} respectively.  
\begin{theorem} \label{TEindata}
Assume that ${g_{z,\delta}}$ satisfies either 
$$ \| F{g_{z,\delta}}  - \phi_z  \|_{L^2(\mathbb{S})}  \rightarrow 0 \quad \text{ or }  \quad    \| N{g_{z,\delta}} - \Phi_{k} (\cdot \,  ,z) \|_{L^2(C)} \rightarrow 0 \quad \text{ as } \, \, \delta \rightarrow 0.$$
Let $k$ be a real transmission eigenvalue and assume that either $\sup_{x \in D} n(x)>1$ or $0<\inf_{x \in D} n(x)<1$. Then for a.e. $z \in D$,  $\left\|{g_{z,\delta}}\right\|_{L^2}$ can not be bounded as $\delta \rightarrow 0$. 
\end{theorem}

%%%%%%%%%%%%%%%%%%%%%%%%%%%%%%%%%%%%%%%%%%%%%%
\noindent{\bf An explicit example for the unit ball in $\R^3$:} We will now give an explicit example where we assume that the $D$ is the unit sphere. We will show that the solution to the far field equation $F{g_{z}}  = \phi_z$ becomes unbounded as $k$ approaches an
interior transmission eigenvalue. We will now assume that the coefficients $n$ and $\eta$ are constant. Separation of variables gives that the interior conductive eigenvalues satisfy 
$$\text{det}\left(
\begin{array}{cc}
 j_p(k \sqrt{n} ) & \, \, -j_p(k ) \\
 k\sqrt{n}  j'_p(k \sqrt{n} )  & \, \,  -k j'_p(k )- \eta j_p(k )
\end{array}
\right)=0 $$ 
which can be written as 
\begin{eqnarray}
k\sqrt{n} j_p'(k\sqrt{n} )j_p(k )-j_p(k\sqrt{n} )\Big(k j_p'(k )+\eta j_p(k ) \Big)=0 \label{eigzero}
\end{eqnarray}
where $j_p$ denotes the spherical Bessel function of the first kind of order $p$.

Now notice that the direct scattering problem \eqref{direct1}--\eqref{direct3} can be written as: find $u\in H^1(D)$ and $u^s \in H^1_{loc}(\R^3 \setminus \overline{D})$ such that 
\begin{eqnarray*}
 \Delta u +k^2 n u=0 \, \, \text{in } \, \, D \quad \text{and} \quad  \Delta u^s + k^2 u^s=0  \, \, \, &\textrm{ in }& \,\,  \R^3 \setminus \overline{D} \\
u^s-u= - \text{e}^{ \text{i}kx \cdot \hat{y} } \quad  \textrm{ and }  \quad  \frac{\partial u^s}{\partial r} +\eta u^s -  \frac{\partial u}{\partial r}= - \left( \frac{\partial }{\partial r} +\eta \right)  \text{e}^{ \text{i}kx \cdot \hat{y}}\quad &\textrm{ on }& \partial D
\end{eqnarray*}
where $u$ is the total field in $D$ and $u^s$ is the radiating scattered field in $\R^3 \setminus \overline{D}$. We make the ansatz that the solutions can be written as the following series 
$$u^s(x)=\sum\limits_{p=0}^{\infty}  \sum\limits_{m=-p}^{p} \alpha^m_p h^{(1)}_p( k |x|) Y^m_p (\hat x) \quad \text{and} \quad u(x)=\sum\limits_{p=0}^{\infty}  \sum\limits_{m=-p}^{p} \beta^m_p  j_p\left(k \sqrt{{n}} |x| \right) Y^m_p (\hat x)$$
where $\hat{x}=x/|x|$. Recall the Jacobi-Anger series expansion for plane waves given by 
$$\text{e}^{\text{i}kx \cdot \hat{y} }= 4 \pi  \sum\limits_{p=0}^{\infty}  \sum\limits_{m=-p}^{p} \text{i}^p j_p( k |x|) Y^m_p (\hat x) \overline{Y^m_p (\hat{y})}.$$
Here we let $h_p^{(1)}$ denote the spherical Hankel function of the first kind of order $p$ and $Y_p^m$ is the spherical harmonic. Applying the boundary condition on $\partial D$ gives the $2 \times 2$ system 
$$\left[
\begin{array}{rr}
	h_p^{(1)}(k ) & -j_p(k\sqrt{n} )\\
	k h_p^{(1)'}(k )+\eta h_p^{(1)}(k )  & - k\sqrt{n} j_p'(k\sqrt{n} )
\end{array}
\right]\left[
\begin{array}{c}
	\alpha_p^m\\
	\beta_p^m
\end{array}
\right]=-4\pi \text{i}^p \overline{Y_p^m({\hat{y}})}\left[
\begin{array}{l}
	j_p(k )\\
	k j_p'(k )+\eta j_p(k )
\end{array}
\right]\,.$$
Solving the system gives that $\alpha_p^m=-4\pi \text{i}^p \overline{Y_p^m({\hat{y}})} \lambda_p$ where 
\begin{eqnarray*}
 \lambda_p= \frac{ k\sqrt{n} j_p'(k\sqrt{n} )j_p(k ) - j_p(k\sqrt{n})\left(k j_p'(k)+\eta j_p(k) \right) }{k\sqrt{n} j_p'(k\sqrt{n})h_p^{(1)}(k)- j_p(k\sqrt{n})\left(k h_p^{(1)'}(k)+\eta h_p^{(1)}(k)\right)}\,.
\end{eqnarray*}
Therefore the far field pattern is given by (see Theorem 2.16 of \cite{coltonkress})
\begin{eqnarray*}
u^\infty(\hat{x}, \hat{y} )=\frac{1}{k}\sum_{p=0}^\infty \sum_{m=-p}^p \frac{1}{\text{i}^{p+1}} \alpha_p^m Y_p^m(\hat{x}) = \frac{4 \pi \text{i} }{k}  \sum\limits_{p=0}^{\infty}  \sum\limits_{m=-p}^{p} \lambda_p Y^m_p (\hat x) \overline{Y^m_p (\hat{y})}.
\end{eqnarray*}
Using the series expansion of the far field pattern we have that by interchanging summation and integration 
\begin{eqnarray*}
({F}g_z)(\hat{x})=\int_{\mathbb S} u^\infty(\hat{x}, \hat{y})g(\hat{y}) \, \text{d}s(\hat{y}) = \frac{4 \pi \text{i} }{k} \sum\limits_{p=0}^{\infty}  \sum\limits_{m=-p}^{p} \lambda_p g_p^m Y^m_p (\hat x)
\end{eqnarray*}
where we define $g^m_p=(g \, , \, Y^m_p (\hat{y}))_{L^2({\mathbb S})}$ the Fourier coefficients of $g$. This implies that the solution to the far field equation is given by 
$$g_z(\hat{y})=  \frac{k}{ \text{i} }  \sum\limits_{p=0}^{\infty}   \sum\limits_{m=-p}^{p} \text{i}^{p} \, \frac{j_p(k|z|)}{\lambda_p} \, \overline{Y^m_p (\hat z)}\, Y^m_p (\hat{y}).$$
Notice that $\lambda_p \rightarrow 0$ as $k \rightarrow k_p$ where $k_p$ is an eigenvalue corresponding to a solution of \eqref{eigzero}. This implies that at least one of the Fourier coefficients $g^m_p$ becomes unbounded which gives that $\left\| g_z \right\|^2_{L^2({\mathbb S})} \rightarrow \infty$ as $k$ approaches an eigenvalue. 

%%%%%%%%%%%%%%%%%%%%%%%%%%%%%%%%%%%%%%%%%%%%%%%%%%%%
\subsection{Convergence of the eigenvalues as $\eta \to 0^+$}\label{convergence}
In this section, we consider the limiting case as $\| \eta(x) \|_{L^{\infty}(\partial D)}\rightarrow 0$ for the eigenvalue problem \eqref{teprob1}--\eqref{teprob2}. 
We will prove that the real eigenvalues $k_\eta$ converge to the eigenvalues $k_0$ where $\eta=0$. We will also consider the convergence of the eigenfunctions. In our analysis, we will focus on the case where $n_{min}>1$ (the analysis for $0<n_{max}<1$ follows from similar arguments).

We consider the transmission eigenvalue problem \eqref{teprobu3}--\eqref{teprobu4} for a nontrivial function in $H^2(D) \cap H^1_0(D)$. Now, define the following bounded sesquilinear forms $\mathcal{A}_{ \eta , k }(\cdot \, , \cdot)$ and $\mathcal{B}(\cdot \, , \cdot)$ on $H^2(D) \cap H^1_0(D) \times H^2(D) \cap H^1_0(D)$ by 
\begin{align}
 \mathcal{A}_{ \eta , k }(u,\varphi) =  \int\limits_D  \frac{1}{ n-1} ( \Delta u +k^2 u )( \Delta \overline{\varphi} +k^2 \overline{\varphi} ) +k^4 u \overline{\varphi}  \, \text{d}x +  \int\limits_{\partial D} \frac{k^2}{\eta} \frac{\partial u}{\partial \nu} \frac{\partial \overline{\varphi}}{\partial \nu} \, \text{d}s \label{bad1} 
\end{align}
and
\begin{align}
\mathcal{B}(u,\varphi) = \int\limits_D \grad u \cdot \grad \overline{\varphi}  \, \text{d}x. \label{bad3}
\end{align}
Therefore, we have that the pair $(k_\eta , u_\eta) \in \R^+ \times H^2(D) \cap H^1_0(D)$ is an eigenpair provided that 
\begin{equation}
\mathcal{A}_{\eta , k_\eta } (u_\eta ,\varphi) -k_\eta^2\mathcal{B}(u_\eta ,\varphi) =0 \quad \text{ for all } \quad  \varphi \in H^2(D) \cap H^1_0(D). \label{bilinearformprob}
\end{equation}
For the case where $\eta = 0$ we define 
\begin{align}
\hspace*{-1cm} \mathcal{A}_{0,k}(u,\varphi) &=  \int\limits_D  \frac{1}{ n-1} ( \Delta u +k^2 u )( \Delta \overline{\varphi} +k^2 \overline{\varphi} ) +k^4 u \overline{\varphi}  \, \text{d}x
\end{align}
and the pair $(k_0 , u_0) \in \R^+ \times H^2_0(D)$ is an eigenpair for $\eta=0$ provided that 
\begin{equation}
\mathcal{A}_{ 0, k_0 } (u_0 ,\varphi) -k_0^2\mathcal{B}(u_0 ,\varphi) =0 \quad \text{ for all } \quad  \varphi \in H^2_0(D)  \label{bilinearformprob2}
\end{equation}
where we define the Sobolev space 
$$H^2_0(D)  = \left\{u\in H^2(D) \, : \, u= \frac{ \partial u}{\partial \nu} = 0 \, \text{ on } \, \, \partial D \right\}$$
equipped with the $H^2(D)$ norm. 

The real interior conductive eigenvalues satisfy 
\begin{equation}
\lambda_j (k , \eta)-k^2=0 \quad \text{ for all} \quad \eta \geq 0 \label{teveq}
\end{equation}
where  
\begin{align}
\lambda_j(k, \eta) = \min\limits_{ U \in \mathcal{U}_j} \max\limits_{u \in U\setminus \{ 0\} }  \frac{ \mathcal{A}_{\eta , k } (u,u) }{  \mathcal{B}(u,u)} \quad \text{ for all} \quad \eta \geq 0. \label{maxmin}
\end{align}
Here $\mathcal{U}_j$ is the set of all $j$-dimensional subspaces of $H^2(D) \cap H^1_0(D)$. By Theorem 5.1 and Corollary 5.3 in \cite{te-cbc} we have that the 
interior conductive eigenvalues $k_\eta$ are an increasing sequence of $\eta$ provided that $\eta$ is decreasing. We now show that the sequence of interior 
conductive eigenvalues are bounded with respect to $\eta>0$.

\begin{theorem}
There exists an infinite sequence of eigenvalues $k^{(j)}_\eta$ for $\eta>0$ with $j \in \N$ satisfying $0 < k^{(j)}_\eta  \leq k^{(j)}_0$ where $k^{(j)}_0$ is an
interior transmission eigenvalue for $\eta=0$. 
\end{theorem}
\begin{proof} 
Just as in Theorem 5.1 in \cite{te-cbc} we have that the Rayleigh quotient \eqref{maxmin} implies that $\lambda_j(k, \eta) \leq \lambda_j(k, 0)$ for all $\eta > 0$. Therefore, we can conclude that
$$\lambda_j(k_0 , \eta) -k_0^2 \leq \lambda_j(k_0 , 0) -k_0^2=0$$ 
where $k_0$ is an interior transmission eigenvalue for $\eta=0$. Recall that the real interior transmission eigenvalues satisfy \eqref{teveq} 
and that the mapping 
$$k \longmapsto \lambda_j(k , \eta) -k^2 \quad \text{for all} \quad \eta \geq 0$$
is continuous for $k \in (0 , \infty)$. 
By Theorem 4.3 in \cite{te-cbc} we have that there exists a $\delta$ (independent of $\eta$) such that for all $k > \delta$ 
$$\lambda_j(k , \eta) -k^2 >0$$
and therefore we can conclude that at least one root of \eqref{teveq} is in the interval $(\delta , k_0]$ for every $j \in \N$, proving the claim. 
\end{proof}

We now have that the sequence of interior conductive  eigenvalues $\{ k_\eta \}_{\eta > 0}$ is a bounded monotonic sequence and is therefore convergent (up to a subsequence if $\eta$ is not decreasing). Now let $ k \in \R^+$ be such that $k_\eta \rightarrow k$ as $\eta \rightarrow 0$. Since the corresponding eigenfunctions $u_\eta$ are nontrival we assume that they are normalized in the $H^1(D)$ norm. In \cite{te-cbc} it is shown that $\mathcal{A}_{\eta , k } (\cdot \, ,\cdot)$ is coercive on $H^2(D) \cap H^1_0(D)$ where the coercivity constant is independent of $k$ and $\eta$. Therefore, we have that there is a constant $\alpha >0$ where 
$$\alpha \| u_\eta \|^2_{H^2(D)} \leq \mathcal{A}_{\eta , k_\eta } (u_\eta ,u_\eta) =k_\eta^2\mathcal{B}(u_\eta ,u_\eta)= k_\eta^2\| \grad u_\eta||^2_{L^2(D) }\leq k^2_0$$
which gives that $u_\eta$ is a bounded sequence in $H^2(D)$ and is therefore weakly convergent to some $u \in H^2(D) \cap H^1_0(D)$ (strongly in $H^1(D)$). Notice, that since $\| u_\eta \|_{H^1(D)}=1$ for all $\eta >0$ the strong $H^1(D)$ convergence implies that the limit has unit $H^1(D)$ norm and therefore $u \neq 0$. 

We now want to show that the limiting pair $(k ,u)$ is a transmission eigenpair for $\eta=0$. Notice that since $u_\eta \in H^2(D) \cap H^1_0(D)$ the compact embedding of $H^{1/2}(\partial D)$ into $L^2(\partial D)$ implies that ${\partial_\nu u_\eta}$ converges strongly to ${\partial_\nu u}$ in $L^2(\partial D)$. By equation \eqref{bilinearformprob} we obtain that  
$$k^2_\eta \int\limits_{\partial D} \frac{1}{ \eta } \left| \frac{\partial u_\eta }{\partial \nu} \right|^2 \, \text{d}s = k_\eta^2\mathcal{B}(u_\eta ,u_\eta) - \mathcal{A}_{0, k_\eta } (u_\eta ,u_\eta).$$
Using the continuity of the sequilinear forms and the fact that both $k_\eta$ and $u_\eta$ are bounded sequences we conclude that 
$$0 \leq \, \, k^2_\eta \left\| \frac{\partial u_\eta }{\partial \nu} \right\|^2_{L^2(\partial D)} \leq C\, \eta_{max} \quad \text{where} \quad \sup\limits_{\partial D} \eta (x) = \eta_{max} >0.$$
Notice that the constant $C$ depends only on $n$, $k_0$ and $\alpha$. 
Notice, that the above inequality gives that ${\partial_\nu u_\eta}$ converges strongly to zero in $L^2(\partial D)$ which implies that $u \in H^2_0(D)$. Therefore, by the convergence of the eigenvalues $k_\eta$ and weak convergence of the eigenfunctions $u_\eta$ if we take any $\varphi \in H^2_0(D)$ 
\begin{align*}
 0 = \lim\limits_{\eta \rightarrow 0}  \mathcal{A}_{\eta, k_\eta} (u_\eta ,\varphi) -k_\eta^2\mathcal{B}(u_\eta ,\varphi) = \, \,  \mathcal{A}_{0 , k} (u ,\varphi) -k^2\mathcal{B}(u ,\varphi).
\end{align*}        
By the above analysis we have the following result. 

\begin{theorem} \label{conv-te-1}
Let $(k_\eta , u_\eta) \in \R^+ \times H^2(D) \cap H^1_0(D)$ be an eigenpair for $\eta > 0$. If $k_\eta$ is bounded with respect to $\eta$ then as $\| \eta(x) \|_{L^{\infty}(\partial D)}\rightarrow 0$ there is a subsequence of $k_\eta$ and $u_\eta$ such that $k_\eta \rightarrow k_0$ and $u_\eta$ converges weakly to $u_0$ in $H^2(D)$ where $(k_0 , u_0) \in \R^+ \times H^2_0(D)$ is an eigenpair for $\eta=0$. 
\end{theorem}
Theorem \ref{conv-te-1} verifies the conjecture in \cite{te-cbc} that the eigenvalues converge as $\eta \rightarrow 0$ to the corresponding eigenvalues for $\eta=0$. In \cite{te-cbc} there are some numerical experiments that indicate that the order of convergence is one.

We will now show that the eigenfunctions converge strongly with respect to the $H^2(D)$ norm. To prove the strong convergence we will use the coercivity of the sequilinear form $\mathcal{A}_{\eta , k } (\cdot \, ,\cdot)$. To this end, let $(k_\eta , u_\eta) \in \R^+ \times H^2(D) \cap H^1_0(D)$ be an
interior transmission eigenpair for $\eta > 0$ and $(k_0 , u_0) \in \R^+ \times H^2_0(D)$ is the limit corresponding to an interior transmission eigenpair for $\eta=0$. Now for $\varphi \in H^2(D) \cap H^1_0(D)$ we have that 
\begin{align*}
\mathcal{A}_{ \eta ,  k_\eta }(u_\eta -u_0, \varphi) &=  \mathcal{A}_{ \eta , k_\eta }(u_\eta , \varphi) -\mathcal{A}_{ 0, k_\eta }(u_0, \varphi) \\
            						&= k^2_\eta \mathcal{B}(u_\eta , \varphi) -k_0^2\mathcal{B}(u_0 , \varphi) + \big( \mathcal{A}_{ 0, k_0 }  -\mathcal{A}_{ 0, k_\eta }\big) (u_0, \varphi)\\
						        &=\big( k^2_\eta -k_0^2\big) \mathcal{B}(u_\eta , \varphi) + k_0^2 \mathcal{B}(u_\eta -u_0 , \varphi) + \big( \mathcal{A}_{ 0, k_0 }  -\mathcal{A}_{ 0, k_\eta }\big) (u_0, \varphi)
\end{align*}
where we have used the variational formulation of the transmission eigenvalue problems \eqref{bilinearformprob} and \eqref{bilinearformprob2} along with the fact that $u_0$ is in $H^2_0(D)$. We now estimate the terms on the right hand side where $\varphi = u_\eta -u_0$. Therefore, simple calculations gives that 
\begin{align*}
 &\big( \mathcal{A}_{ 0, k_0 }  -\mathcal{A}_{ 0, k_\eta }\big) (u_0, \varphi) \\
 &\hspace{0.6in}=  \big(k^2_0 -k_\eta^2 \big) \int\limits_D \frac{1}{n-1} (  \overline{\varphi} \, \Delta u_0 + u_0 \, \Delta \overline{\varphi}) \, \text{d}x +   \big(k^4_0 -k_\eta^4 \big) \int\limits_D \frac{n}{ n-1 } u_0 \, \overline{\varphi} \, \text{d}x.
\end{align*}
Now, letting $\varphi=u_\eta -u_0$ and using the fact that $u_0$ and $u_\eta$ are bounded with respect to $\eta$ gives that 
$$ \Big| \big( \mathcal{A}_{ 0, k_0 }  -\mathcal{A}_{ 0, k_\eta }\big) (u_0, u_\eta -u_0) \Big| \leq C_1 \big| k^2_0 -k_\eta^2\big| + C_2 \big| k^4_0 -k_\eta^4\big| \longrightarrow 0 \quad \text{as} \, \, \, \, \eta \rightarrow 0$$
where the positive constants $C_1$ and $C_2$ are independent of $\eta$. Similarly we can conclude that 
$$  \Big|  ( k^2_\eta -k_0^2) \mathcal{B}(u_\eta , u_\eta -u_0) \Big|  \leq C_3 \big| k^2_0 -k_\eta^2 \big|  \longrightarrow 0 \quad \text{as} \, \, \, \, \eta \rightarrow 0$$
where the positive constant $C_3$ is independent of $\eta$. Using the fact that $u_\eta -u_0$ converges strongly to zero in $H^1(D)$ gives that  
$$ k_0^2 \mathcal{B}(u_\eta -u_0 , u_\eta -u_0)=k_0^2 \| \grad (u_\eta - u_0) \|^2_{L^2(D)}  \longrightarrow 0 \quad \text{as} \, \, \, \, \eta \rightarrow 0.$$
By using the coercivity of $\mathcal{A}_{\eta , k } (\cdot \, ,\cdot)$ we obtain that 
\begin{align*}
\hspace*{-0.2in} \alpha \| u_\eta -u_0 \|^2_{H^2(D)}  & \leq \mathcal{A}_{ \eta ,  k_\eta }(u_\eta -u_0, u_\eta -u_0) \\
						        &\hspace{-0.8in}= \big( \mathcal{A}_{ 0, k_0 }  -\mathcal{A}_{ 0, k_\eta }\big) (u_0, u_\eta -u_0) + \big( k^2_\eta -k_0^2\big) \mathcal{B}(u_\eta , u_\eta -u_0) + k_0^2 \mathcal{B}(u_\eta -u_0 , u_\eta -u_0)
\end{align*}
giving that $u_\eta -u_0$ converges strongly to zero in $H^2(D)$. 

\begin{theorem} \label{conv-te-2}
Let $(k_\eta , u_\eta) \in \R^+ \times H^2(D) \cap H^1_0(D)$ be an eigenpair for $\eta > 0$. If $k_\eta$ is bounded with respect to $\eta$ then as $\| \eta(x) \|_{L^{\infty}(\partial D)}\rightarrow 0$ there is a subsequence of $k_\eta$ and $u_\eta$ such that $k_\eta \rightarrow k_0$ and $u_\eta$ converges strongly to $u_0$ in $H^2(D)$ where $(k_0 , u_0) \in \R^+ \times H^2_0(D)$ is an eigenpair for $\eta=0$. 
\end{theorem}

%%%%%%%%%%%%%%%%%%%%%%%%%%%%%%%%%%%%%%
\subsubsection{Convergence rate for radially symmetric eigenfunctions}
We now look at the radially symmetric eigenfunctions and prove the linear convergence rate as predicted by our numerical examples (see Section \ref{numerics}). Notice that 
$$w_\eta (r)=  \text{J}_0(k_\eta) \text{J}_0(k_\eta \sqrt{n} r) \quad \text{ and } \quad v_\eta (r)=  \text{J}_0(k_\eta \sqrt{n}) \text{J}_0(k_\eta r)$$
are the eigenfunctions for the unit circle with $k_\eta \in \R$ the corresponding eigenvalues. We will prove the linear convergence rate for $w_\eta$ and the analysis is similar for the other eigenfunction. 
To this end, we assume that the eigenvalues have a linear convergence rate and we therefore have  
\begin{eqnarray*}
&&\hspace{-0.5in}| w_\eta (r) - w_0(r) |\\ 
&=&  \big| \text{J}_0(k_\eta) \text{J}_0(k_\eta \sqrt{n} r) -\text{J}_0(k_0) \text{J}_0(k_0 \sqrt{n} r) \big| \\
&\leq& \big| \text{J}_0(k_\eta) \text{J}_0(k_\eta \sqrt{n} r) -\text{J}_0(k_0) \text{J}_0(k_\eta \sqrt{n} r) \big| +\big| \text{J}_0(k_0) \text{J}_0(k_\eta \sqrt{n} r) -\text{J}_0(k_0) \text{J}_0(k_0 \sqrt{n} r) \big| \,,
\end{eqnarray*}
where $k$ is a transmission eigenvalue for $\eta=0$. 
Since $\text{J}_0$ is bounded by one and the Mean Value Theorem
\begin{align*}
| w_\eta (r) - w_0(r) | & \leq \big| \text{J}_0(k_\eta) -\text{J}_0(k_0)  \big| +\big|  \text{J}_0(k_\eta \sqrt{n} r) - \text{J}_0(k_0 \sqrt{n} r) \big| \\
&\leq  \big| \text{J}_1(\xi_1)\big| \big| k_\eta -k_0 \big| +\sqrt{n} \, r\,  \big| \text{J}_1(\xi_2)\big| \big| k_\eta -k_0 \big| \,,
\end{align*}
where $k_\eta < \xi_1 < k_0$ and $ k_\eta \sqrt{n} \, r< \xi_2 < k_0 \sqrt{n}\,  r$ with $0 \leq r \leq 1$. Similarly, since $\text{J}_1$ is bounded by one, we have that 
$$| w_\eta (r) - w_0(r) |  \leq C \big| k_\eta -k_0 \big| $$
where $C>0$ is a positive constant that is independent of $\eta$. Now, since we have that the eigenvalues have a linear convergence rate, we can conclude that the eigenfunctions also have a linear convergence rate. Also note that the case for spherical eigenfunctions in $\R^3$ works similarly.

We now consider the case of complex eigenvalues so we assume that $k_\eta \in \C$. Before we begin, we need a generalization of the Mean Value Theorem for holomorphic functions. To this end, assume that $f(z)$ is a holomorphic function and let $a$ and $b$ be distinct points. Now define the real valued functions  
$$ F(t) = \text{Re} \Big(\overline{(b-a)} f \big( a+t(b-a) \big) \Big) \quad \text{and} \quad G(t) = \text{Im} \Big(\overline{(b-a)} f \big( a+t(b-a) \big) \Big)$$
for $0 \leq t \leq 1$. Applying the Mean Value Theorem to the functions $F$ and $G$ implies that there exits $\alpha$, $\beta$ on the line segment connecting $a$ and $b$ such that 
$$|b-a|^2 \text{Re} \big(  f'(\alpha) \big) = \text{Re} \Big(  \overline{(b-a)} \big( f(b) -f(a) \big) \Big)$$
and 
$$|b-a|^2 \text{Im} \big(  f'(\beta) \big) = \text{Im} \Big(  \overline{(b-a)} \big( f(b) -f(a) \big) \Big).$$
Now recall that 
\begin{align*}
w_\eta (r) - w_0(r) &=   \text{J}_0(k_\eta) \text{J}_0(k_\eta \sqrt{n} r) -\text{J}_0(k_0) \text{J}_0(k_0 \sqrt{n} r)  \\
& \hspace{-0.4in}= \text{J}_0(k_\eta \sqrt{n} r) \Big( \text{J}_0(k_\eta) -\text{J}_0(k_0)  \Big) + \text{J}_0(k_0) \Big(  \text{J}_0(k_\eta \sqrt{n} r) - \text{J}_0(k_0 \sqrt{n} r) \Big). 
\end{align*}
Therefore, by the generalization of the Mean Value Theorem
\begin{align}
|k_\eta-k_0|^2 \text{Re} \big( \text{J}_1(\alpha) \big) = \text{Re} \Big(  \overline{(k_0-k_\eta )} \big( \text{J}_0(k_\eta) -\text{J}_0(k_0)\big) \Big) \label{complex-roc1}
\end{align}
and 
\begin{align}
|k_\eta-k_0|^2 \text{Im} \big(  \text{J}_1(\beta) \big) = \text{Im} \Big(  \overline{(k_0-k_\eta)} \big( \text{J}_0(k_\eta) -\text{J}_0(k_0) \big) \Big) \label{complex-roc2}
\end{align}
where $\alpha$, $\beta$ are on the line segment connecting $k_\eta$ and $k_0$. Since $\text{J}_0$ is analytic this implies that the right hand sides of equations \eqref{complex-roc1} and \eqref{complex-roc2} are of order $|k_\eta-k_0|^2 =\mathcal{O}(\eta^2)$ and therefore so is the left hand side which gives that $\big|  \text{J}_0(k_\eta) -\text{J}_0(k_0)  \big| = \mathcal{O}(\eta)$.  Similarly one can show that $\big|  \text{J}_0(k_\eta \sqrt{n} r) - \text{J}_0(k_0 \sqrt{n} r) \big| = \mathcal{O}(\eta)$. Therefore, we have that 
$$| w_\eta (r) - w_0(r) |   = \mathcal{O}(\eta)$$
for any $0 \leq r \leq 1$ proving the linear convergence that is seen in Section \ref{numerics}. This analysis also works in $\R^3$. 
%provided that $k_\eta$ and $k_0$ are contained in a simply connected region of the complex plane where $j_0$ is analytic. 

The main assumption in this section is that the eigenvalues converge linearly as $\eta \rightarrow 0$. We now give some analytical evidence that the convergence rate is linear for the eigenvalues corresponding to the radially symmetric eigenfunctions. Recall, that for the unit sphere in $\R^3$ that the eigenvalues are the roots of 
$$d_\eta (k) = k\sin(k \sqrt{n} ) \cos( k ) - k\sqrt{n} \sin(k) \cos( k \sqrt{n}) + \eta \sin(k) \sin(k\sqrt{n}).$$
Therefore, we define that function 
$$ f(k)= k \cot (k) -k \sqrt{n} \cot (k \sqrt{n}) $$
where it is clear that $k_\eta$ is an eigenvalue if $f(k_\eta)=\eta$ and $k_0$ is an eigenvalue provided that $f(k_0)=0$. Now, assume that $k_0$ is contained in an interval where $f$ is a $C^1$ function and that $f'(k_0) \neq 0$, by the Inverse Function Theorem there is an open set containing $k_0$ where $f$ has a $C^1$ inverse. We denote the inverse of $f$ by $g$ which implies that $k_\eta = g(\eta)$. Therefore, appealing to the Mean Value Theorem we can conclude that $| k_\eta -k_0| = |g'(\xi)| \,  \eta$ for all $\eta$ sufficiently small where $k_\eta < \xi <k_0$. Since $g$ is $C^1$ in an interval we can conclude that $|k_\eta-k_0 | =\mathcal{O}(\eta)$ for all $\eta>0$ sufficiently small.

%%%%%%%%%%%%%%%%%%%%%%%%%%%%%%%%%%%%%%
\subsection{ Existence and discreteness for an absorbing media }
In this section, we consider the case of an absorbing media. To this end, we assume that the material coefficients are given by 
$$n(x)= n_1(x) +\text{i} \frac{ n_2(x)}{k} \quad \text{and} \quad \eta(x)=\eta_1(x) +\text{i}\eta_2(x)$$
where $n_\ell (x)$ are real-valued positive functions in $L^{\infty}(D)$ for $\ell=1,2$. For the conductivity parameter $\eta$ we assume that the functions $\eta_\ell (x)$ are real-valued positive functions in $L^{\infty}(\partial D)$. The analysis of the transmission eigenvalue problem (i.e. $\eta=0$) has already been studied in \cite{TE-absorbing}. In this section, we consider the existence and discreteness of the conductive eigenvalues for an absorbing media. We now rewrite the variational formulation \eqref{TE-varform} of the eigenvalue problem. Manipulating the variational form in \eqref{TE-varform} we see that 
\begin{align*}      
&0=  \int\limits_D \frac{1}{n-1} \Delta u \,  \Delta \overline{\varphi}  \,  \text{d}x - \int\limits_D \frac{k^2}{n-1} (  \overline{\varphi} \, \Delta u + u \, \Delta \overline{\varphi}) \,  \text{d}x +  k^2 \int\limits_D \grad u \cdot \grad \overline{\varphi} \,  \text{d}x \\
&\qquad \qquad \qquad \qquad+k^4  \int\limits_D \frac{n}{ n-1 } u \, \overline{\varphi} \,  \text{d}x+  \int\limits_{\partial D} \frac{k^2}{\eta} \frac{\partial u}{\partial \nu} \frac{\partial \overline{\varphi} }{\partial \nu} \,  \text{d}s
\end{align*}
for all $\varphi \in H^2(D) \cap H^1_0(D)$. 
Substituting $n= n_1 +\text{i} \, { n_2}/{k}$ yields the following 
\begin{align*}      
&0=  \int\limits_D \frac{k}{k(n_1-1) +\text{i} n_2}\,  \Delta u \,  \Delta \overline{\varphi}  \, \text{d}x + k^2 \int\limits_{\partial D} \frac{1}{\eta} \frac{\partial u}{\partial \nu} \frac{\partial \overline{\varphi} }{\partial \nu} \, \text{d}s+  k^2 \int\limits_D \grad u \cdot \grad \overline{\varphi} \,  \text{d}x \\
&\qquad \qquad \qquad  + k^4  \int\limits_D \frac{k n_1 + \text{i} n_2 }{ k(n_1-1) +\text{i} n_2 }\,  u \, \overline{\varphi} \,  \text{d}x - k^2 \int\limits_D  \frac{k}{k(n_1-1) +\text{i} n_2}(  \overline{\varphi} \, \Delta u + u \, \Delta \overline{\varphi}) \,  \text{d}x.  
\end{align*}
%Since $n_\ell \in L^{\infty}(D)$ for $\ell = 1,2$
Notice that we require that $| k(n_1-1) +\text{i} n_2 |$ is bounded below, so we assume that $n_1-1 \geq \gamma_1>0$ and $n_2\geq \gamma_2 >0$ and that there is a 
$$\delta >0 \quad \text{ such that} \quad \delta < \inf\limits_{x \in D} \left\{\frac{ n_2(x)}{n_1(x)-1} \right\}.$$
For all $k$ in the set   
$$ G_\delta= \big\{ z \in \C \quad \text{such that } \quad  \mathrm{Im}(z)>-\delta  \big\}$$
we have that $|k(n_1-1) +\text{i} n_2 |$ is bounded below.

The transmission eigenvalue problem can now be written in operator form as: find the values $k \in G_\delta \subset \C$ such that there is a nontrivial solution $u \in H^2(D) \cap H^1_0(D)$ satisfying
$$ (\mathbb{A}_k+ k \mathbb{B}_k ) u =0$$
where the bounded linear operators $\mathbb{A}_k$ and $\mathbb{B}_k: H^2(D) \cap H^1_0(D) \mapsto H^2(D) \cap H^1_0(D)$ are defined by the Riesz representation theorem such that 
\begin{align*}
&(\mathbb{A}_k u, \varphi )_{H^2(D)}= \int\limits_D \frac{1}{k(n_1-1) +\text{i} n_2} \Delta u \,  \Delta \overline{\varphi}  \,  \text{d}x,\\
&(\mathbb{B}_k u, \varphi )_{H^2(D)}=  \int\limits_{\partial D} \frac{1}{\eta} \frac{\partial u}{\partial \nu} \frac{\partial \overline{\varphi} }{\partial \nu} \, \text{d}s+   \int\limits_D \grad u \cdot \grad \overline{\varphi} \,  \text{d}x + k^2  \int\limits_D \frac{k n_1 + \text{i} n_2 }{ k(n_1-1) +\text{i} n_2 }\,  u \, \overline{\varphi} \,  \text{d}x \\
&\qquad \qquad \qquad   \qquad \qquad - k \int\limits_D  \frac{1}{k(n_1-1) +\text{i} n_2} \, (  \overline{\varphi} \, \Delta u + u \, \Delta \overline{\varphi}) \,  \text{d}x
\end{align*}
for all $\varphi \in H^2(D) \cap H^1_0(D)$. We see that the operator $\mathbb{B}_k$ is compact by appealing to Rellich's embedding 
theorem and the compact embedding of $H^{1/2}(\partial D)$ into $L^2(\partial D)$. 

To prove the discreteness we will use the analytic Fredholm Theorem. To this end, notice that the mappings $k \mapsto \mathbb{A}_k$ and $\mathbb{B}_k$ 
are analytic for all $k \in G_\delta$. Next, we show that $\mathbb{A}_k$ is coercive for all $k \in G_\delta$, therefore let $k=\text{Re}(k)+\text{i} \,\text{Im}(k)$ which gives that 
$$k(n_1-1) +\text{i} n_2 =\text{Re}(k)(n_1-1) +\text{i}\big[ n_2 +\text{Im}(k)(n_1-1)\big].$$
This implies that the sequilinear form for $\mathbb{A}_k$ is given by 
$$(\mathbb{A}_k u, \varphi )_{H^2(D)}= \int\limits_D \frac{\text{Re}(k)(n_1-1) -\text{i} \big[ n_2 +\text{Im}(k)(n_1-1)\big] }{ |\text{Re}(k)(n_1-1)|^2+| n_2 +\text{Im}(k)(n_1-1)|^2} \Delta u \,  \Delta \overline{\varphi}  \,  \text{d}x.$$
Now, we obtain 
\begin{align*}
-\mathrm{Im} (\mathbb{A}_k u, u )_{H^2(D)}=\int\limits_D \frac{\big[ n_2 +\text{Im}(k)(n_1-1)\big] }{ |\text{Re}(k)(n_1-1) |^2+|n_2 +\text{Im}(k)(n_1-1)|^2} | \Delta u|^2   \,  \text{d}x
\end{align*}
and since $k \in G_\delta$ we can conclude that $n_2 +\text{Im}(k)(n_1-1)$ is bounded below. Recall, that the $L^2(D)$-norm of the Laplacian is equivalent to the $H^2(D)$-norm in $H^2(D) \cap H^1_0(D)$, giving that $\mathbb{A}_k$ is coercive for all $k \in G_\delta$. 
We now have all we need to prove the discreteness of the transmission eigenvalues in the set $G_\delta \subset \C$. 

\begin{theorem}
Assume that $n_1-1 \geq \gamma_1>0$ and $n_2\geq \gamma_2 >0$ and that  
$$\delta >0 \quad \text{ is such that } \quad \delta < \inf\limits_{x \in D} \left\{\frac{ n_2(x)}{n_1(x)-1} \right\}.$$
Then the set of interior transmission eigenvalues is at most discrete in the upper half of the complex plane.
\end{theorem}
\begin{proof}
We have that for all $k \in G_\delta$ that $\mathbb{A}^{-1}_k$ exists as a bounded linear operator from $H^2(D) \cap H^1_0(D)$ to itself. Since $\mathbb{A}_k$ depends analytically on $k$ in $G_\delta$ we can conclude that $\mathbb{A}^{-1}_k$ is analytic. This gives that $k$ is an interior conductive eigenvalue if and only if 
$\mathbb{I}+k \mathbb{A}^{-1}_k \mathbb{B}_k$ has a nontrivial kernel. Since the mapping $k \mapsto \mathbb{A}^{-1}_k \mathbb{B}_k$ is analytic and compact we only need to prove that $\mathbb{I}+k \mathbb{A}^{-1}_k \mathbb{B}_k$ is injective at some point in $G_\delta$ by the analytic Fredholm Theorem. To this end, notice that  for $k=0$ then the operator $\mathbb{I}+k \mathbb{A}^{-1}_k \mathbb{B}_k $ is the identity and is therefore injective, proving the claim since the upper half of the complex plane is a subset of $G_\delta$.  
\end{proof}

The last result gives the discreteness of the conductive eigenvalues and now we turn our attention to existence. The idea is to consider the case of an absorbing material as a perturbation of a non-absorbing material and consider what happens in the limit as the absorption tends to zero. To start with, we will use the results in \cite{sharinonlinear} where the perturbation of a non-linear eigenvalue problem is studied. We now consider the case where Re$(k)>0$ and  
\begin{align*}
\mathrm{Re} (\mathbb{A}_k u, u )_{H^2(D)} = \int\limits_D  \frac{\text{Re}(k)(n_1-1)}{ |\text{Re}(k)(n_1-1)|^2+| n_2 +\text{Im}(k)(n_1-1)|^2} | \Delta u|^2   \,  \text{d}x,
\end{align*}
and notice that if $n_1 - 1$ is uniformly positive (or negative) then the operators $\mathbb{A}_k$ is coercive for all $k$ such that Re$(k)>0$. We can now write our eigenvalue problem as 
$$k  \mathbb{T}(k) u = u \quad \text{ where } \quad \mathbb{T}(k)= - \mathbb{A}^{-1}_k \mathbb{B}_k.$$
It is clear that the operator $\mathbb{T}(k)$ depends on the absorption parameters $n_2$ and $\eta_2$.  We first consider the case for where $\eta$ is real-valued. To this end, let $\mathbb{T}_{n_2}(k)=  \mathbb{A}^{-1}_{n_2 , k} \mathbb{B}_{n_2 , k}$ be the operator associated with our non-linear eigenvalue problem. Theorem 3.1 in \cite{sharinonlinear} gives the existence of conductive eigenvalues for $n_2$ sufficiently small provided that $\mathbb{T}_{n_2}(k)$ converges in the operator norm to $\mathbb{T}_{0}(k)$ as $n_2 \to 0^+$ in $L^{\infty}(D)$ for all $k$ such that Re$(k)>0$. 
%To prove the convergence of the operator we study the convergence of $\mathbb{A}^{-1}_{n_2 , k}$ and  $\mathbb{B}_{n_2 , k}$ as $n_2 \to 0^+$ in $L^{\infty}(D)$.

A follow result gives that convergence of $\mathbb{B}_{n_2 , k}$ which is a simple consequence of the Cauchy-Schwartz inequality and the $L^\infty$ convergence of the function $n_2$ to the zero function. 
\begin{theorem}
If $n_1 - 1 $ is uniformly positive (or negative) then the operator $\mathbb{B}_{n_2 , k}$ converges in norm to $\mathbb{B}_{0 , k}$ as $n_2 \to 0$ in $L^{\infty}(D)$ for all $k$ such that Re$(k)>0$. 
\end{theorem}

Next, we consider the convergence of the operator $\mathbb{A}^{-1}_{n_2 , k}$  as $n_2 \to 0$ in $L^{\infty}(D)$. To prove the convergence we first notice that $\big\| \mathbb{A}^{-1}_{n_2 , k} \big\|$ is bounded as $n_2$ tends to zero for Re$(k)>0$. Indeed, assume that $n_2 <1$ then we have that 
$$\sigma(n_1,k)  \| \Delta u\|^2_{L^2(D)} \, \leq \,   \mathrm{Re} (\mathbb{A}_{n_2 , k} u, u )_{H^2(D)} $$
where the coercivity constant 
$$\sigma(n_1,k) = \inf\limits_{x \in D}  \frac{\text{Re}(k)(n_1-1)}{ |\text{Re}(k)(n_1-1)|^2+| 1 +\text{Im}(k)(n_1-1)|^2} $$
is independent of $n_2$.
 Therefore, we have that for $n_2 < 1$ that $\big\| \mathbb{A}^{-1}_{n_2 , k} \big\| $ is uniformly bounded.  

\begin{theorem}
If $n_1 - 1 $ is uniformly positive (or negative) then the operator $\mathbb{A}^{-1}_{n_2 , k} $ converges in norm to $\mathbb{A}^{-1}_{0 , k}$ as $n_2 \to 0$ in $L^{\infty}(D)$ for all $k$ such that Re$(k)>0$. 
\end{theorem}
\begin{proof}
We begin by letting $w_{n_2} = \mathbb{A}^{-1}_{n_2 , k} f$ and $w= \mathbb{A}^{-1}_{0 , k} f$ for some $f \in H^2(D) \cap H^1_0(D)$. We wish to show that $w_{n_2}$ converges to $w$ in the $H^2(D)$ as $n_2 \to 0$ in $L^{\infty}(D)$. To do so, notice that 
\begin{align*}
\Big( \mathbb{A}_{n_2 , k}  \left(  w_{n_2} - w \right)  , \varphi \Big)_{H^2(D)} & = \Big(  \left(  \mathbb{A}_{0 , k} - \mathbb{A}_{n_2 , k}\right)w  , \varphi \Big)_{H^2(D)}\\
														     &= \int\limits_{D} \left[ \frac{1}{k(n_1 -1 )}  - \frac{1}{k(n_1 -1) +\text{i} n_2 }  \right]  \Delta w \Delta \overline{\varphi} \, dx.
\end{align*}
Appealing to the fact that $\mathbb{A}_{n_2 , k}$ is coercive gives that 
$$\sigma \| w_{n_2} - w  \|_{H^2(D)} \leq \left\| \frac{1}{k(n_1 -1 )+\text{i} n_2}  - \frac{1}{k(n_1 -1)}  \right\|_{L^{\infty}(D)} \| \Delta w \|_{L^2(D)}$$
for $\varphi = w_{n_2} - w$. Now, since $\mathbb{A}^{-1}_{0 , k}$ is a bounded operator we have that 
$$ \Big\| \big(\mathbb{A}^{-1}_{n_2 , k}   - \mathbb{A}^{-1}_{0 , k}  \big) f  \Big\|_{H^2(D)} \leq  C \left\| \frac{1}{k(n_1 -1 )+\text{i} n_2}  - \frac{1}{k(n_1 -1)}  \right\|_{L^{\infty}(D)} \| f \|_{H^2(D)}$$
proving the claim by taking the supremum over $f \in H^2(D)$ with unit norm. 
\end{proof}

This gives that the operator $\mathbb{T}_{n_2}(k)$ converges in norm to $\mathbb{T}_{0}(k)$ as $n_2 \to 0$ in $L^{\infty}(D)$ for all $k$ such that Re$(k)>0$. Where the operator $\mathbb{T}_{0}(k)$ corresponds to the conductive eigenvalue problem with a real-valued refractive index. In \cite{te-cbc} it is proven that there exists infinity many conductive eigenvalues for real-valued refractive index and boundary conductivity. Now appealing to Theorem 3.1 in \cite{sharinonlinear} we have that for any $n_2$ sufficiently small there is a eigenvalue corresponding to an absorbing material in a neighborhood of an eigenvalue of a non-absorbing media which gives the following result. 

\begin{theorem}\label{eig-exist1}
Assume that $\eta_2 = 0$ and that $n_1 - 1 $ is uniformly positive (or negative). Then there exists infinity many conductive eigenvalues $k$ such that Re$(k)>0$ provided that $n_2$ is sufficiently small. 
\end{theorem}

For the case where $\eta_2 \neq 0$ one can easily show that the operator $\mathbb{B}_{\eta_2 , k}$ converges in norm to the corresponding operator as $\eta_2 \to 0$ in $L^{\infty}(\partial D)$. Therefore, we can conclude that the operator $\mathbb{T}_{\eta_2}(k)=  \mathbb{A}^{-1}_{ k} \mathbb{B}_{\eta_2 , k}$ also converges in norm as $\eta_2 \to 0^+$ in $L^{\infty}(\partial D)$. Now by combining Theorem \ref{eig-exist1} and Theorem 3.1 in \cite{sharinonlinear} gives the following result. 

\begin{theorem}\label{eig-exist2}
Assume that $n_1 - 1 $ is uniformly positive (or negative). Then there exists infinity many transmission eigenvalues $k$ such that Re$(k)>0$ provided that $n_2$ and $\eta_2$ are sufficiently small. 
\end{theorem}

%%%%%%%%%%%%%%%%%%%%%%%%%%%%%%%%%%%%%%%%%%%%%%%%%%%%%%%%%%%%%%%%%%%%%%%%
\section{Numerical experiments}\label{numerics}
\subsection{Convergence of the eigenvalues and eigenfunctions}
As we have proven in Section \ref{convergence}, the interior transmission eigenvalues (ITEs) $k_\eta$ converge 
to the ITEs $k_0$ as $\eta$ approach $0$. 
This is true for both real and complex-valued ITEs. 
For the index of refraction $n=3$ and the sequence $\eta=1/2^i$ with $i=0,1,\ldots,8$, we obtain two real and three complex-valued ITEs 
in the domain $\Omega=[0,10]\times [0,10]\mathrm{i}\subset \mathbb{C}$ using a unit sphere. 
We also assume that the eigenfunctions are radially-symmetric, therefore $w(r)$ and $v(r)$ with  $r=|x|$. This implies that
\begin{align*}
r^2 w'' +2 rw'+k^2 r^2 nw=0 \quad  \text{and}  \quad r^2 v'' +2 rv'+k^2 r^2 v=0 \quad \textrm{ in } \,  D. 
\end{align*} 
It is easy to see that $w(r)=c_1 j_0(k \sqrt{n} r)$ and $v(r)=c_2 j_0(k r)$ where $j_0$ is the spherical Bessel function of the first kind of order zero with $c_1$ and $c_2$  constants.  Applying the boundary conditions we have that $k$ is a 
ITE if and only if \\
$$\text{det}\left(
\begin{array}{cc}
 j_0(k \sqrt{n} ) & \, \,  -j_0(k ) \\
 k\sqrt{n}  j'_0(k \sqrt{n} )  & \, \,  -k j'_0(k )- \eta j_0(k )
\end{array}
\right)=0. $$ 
Now recall that 
$$ j_0( t )=\frac{\sin t}{t} \quad \text{ and } \quad j'_0( t )= \frac{-\sin t}{t^2} +\frac{\cos t}{t} $$
which gives that $k$ is an ITE provided that $k$ is a root of 
$$ d_\eta (k) = k\sin(k \sqrt{n} ) \cos( k ) - k\sqrt{n} \sin(k) \cos( k \sqrt{n}) + \eta \sin(k) \sin(k\sqrt{n}).$$
It has been proven in \cite{coltonyj} that the function $d_0(k)$ has infinitely many complex roots provided that $\sqrt{n}$ is not an integer or a reciprocal of an integer. By appealing to Hurwitz's theorem (see p. 152 of \cite{conway}) that for $\eta$ sufficiently small there exists 
infinitely many complex roots of $d_\eta(k)$. 

\begin{theorem}\label{complex-te-1}
Assume that $n$ and $\eta$ are constants where $\sqrt{n}$ is not an integer or a reciprocal of an integer. Then there exists infinitely many complex-valued ITEs for the sphere in $\R^3$, provided that $\eta$ is sufficiently small. 
\end{theorem}

\begin{table}[!ht]
\centering
\begin{tabular}{c|c|c|c|c}
 4.443358    & 8.328\,578   & $3.003079+0.723476\mathrm{i}$&  $6.305573+0.787309\mathrm{i} $  &  $ 9.598536+0.669770\mathrm{i}$  
\end{tabular}
\caption{The eigenvalues for the unit sphere for $n=3$ and $\eta= 0$.}
\end{table}
%$4.443\,358$, $8.328\,578$, $3.003\,079+0.723\,476\mathrm{i}$, $6.305\,573+0.787\,309\mathrm{i}$,
%$9.598\,536+0.669\,770\mathrm{i}$. 
With the absolute error $\epsilon_\eta^{(j)}=|k_0^{(j)}-k_\eta^{(j)}|$ 
for the $j$-th ITE we define the estimated order of convergence for the $j$-th ITE by 
$\mathrm{EOC}^{(j)}=\log(\epsilon_\eta^{(j)}/\epsilon_{\eta/2}^{(j)})/\log(2)$.

In Table \ref{tableConvergenceSphere}, we show the estimated order of convergence for the five ITEs.
\begin{table}[!ht]
\centering
\begin{tabular}{c|ccccc}
$\eta$    & $\mathrm{EOC}^{(1)}$ & $\mathrm{EOC}^{(2)}$ & $\mathrm{EOC}^{(3)}$ & $\mathrm{EOC}^{(4)}$ & $\mathrm{EOC}^{(5)}$\\
 \hline
 1    &     &     &     &     &\\
 1/2  &0.977&1.023&1.068&1.007&0.991\\
 1/4  &0.993&1.013&1.044&1.006&0.997\\
 1/8  &0.998&1.007&1.024&1.004&0.999\\
 1/16 &0.999&1.003&1.012&1.002&0.999\\
 1/32 &1.000&1.002&1.006&1.001&1.000\\
 1/64 &1.000&1.001&1.003&1.001&1.000\\
 1/128&1.000&1.000&1.002&1.000&1.000\\
 1/256&1.000&1.000&1.001&1.000&1.000\\
 \hline
\end{tabular}
\caption{\label{tableConvergenceSphere}The estimated order of convergence for five ITEs for a unit sphere using $n=3$ as $\eta\rightarrow 0$.}
\end{table}

Interestingly, the order of convergence seems to be linear for both the real and complex-valued ITEs. Note that the conjecture of the 
linear convergence for constant $n$ has already been made in \cite[Table 3]{te-cbc} for the real-valued ITEs, but the proof was still open.

Next, we show numerically that the 
eigenfunctions $w_\eta$ and $v_\eta$ for the ITEs converge to the eigenfunctions $w_0$ and $v_0$, respectively. 
Therefore, we need the squared $L^2$ error for the unit sphere which is given as
\[
\epsilon_\eta^{(w)}=4\pi\int_0^1 \left|w_\eta(r)-w_0(r)\right|^2r^2\,\mathrm{d}r 
\]
with the obvious definition of $\epsilon_\eta^{(v)}$. Note that we used $c_1=j_0(k_\eta)$ and $c_2=j_0(k_\eta\sqrt{n})$ in the definition of the two eigenfunctions. 
The estimated order of convergence is given by 
$$\mathrm{EOC}^{(w)}=\log\left(\sqrt{\epsilon_\eta^{(w)}/\epsilon_{\eta/2}^{(w)}}\right)/\log(2)=\log \big(\epsilon_\eta^{(w)}/\epsilon_{\eta/2}^{(w)} \big)/(2\log(2))$$ and 
likewise $\mathrm{EOC}^{(v)}$. In Table 
\ref{tableConvergenceSphereEigenfunctions} we show the estimated order of convergence for the first real and the first complex-valued ITEs.
\begin{table}[!ht]
\centering
\begin{tabular}{c|cc|cc}
      $\eta$    & $\mathrm{EOC}^{(w)}$ & $\mathrm{EOC}^{(v)}$ & $\mathrm{EOC}^{(w)}$ & $\mathrm{EOC}^{(v)}$ \\
 \hline
 1    &     &     &     &     \\
 1/2  &0.992&1.004&1.093&1.075\\
 1/4  &1.002&1.000&1.071&1.063\\
 1/8  &1.002&0.999&1.040&1.036\\
 1/16 &1.001&1.000&1.021&1.019\\
 1/32 &1.001&1.000&1.010&1.010\\
 1/64 &1.000&1.000&1.005&1.005\\
 1/128&1.000&1.000&1.003&1.002\\
 1/256&1.000&1.000&1.001&1.001\\
 \hline
\end{tabular}
\caption{\label{tableConvergenceSphereEigenfunctions}The estimated order of convergence for two pairs of  eigenfunctions for a unit sphere 
using $n=3$ as $\eta\rightarrow 0$.}
\end{table}
As we can see the estimated order of convergence is linear. The proof was still open.

Next, we use a unit circle with the same parameters as before. Therefore, the radially symmetric eigenfunctions $w(r)$ and $v(r)$ with $r=|x|$, satisfy the differential equations 
\begin{align*}
r^2 w'' +rw'+k^2 r^2 nw=0 \quad  \text{and}  \quad r^2 v'' +rv'+k^2 r^2 v=0 \quad \textrm{ in } \,  D. 
\end{align*} 
This implies that the eigenfunctions are given by $w(r)= c_1 \text{J}_0(k \sqrt{n} r)$ and $v(r)=c_2 \text{J}_0(k r)$ where $\text{J}_0$ is the Bessel function of order zero with $c_1$ and $c_2$ constants. Applying the boundary conditions and using that $\text{J}'_0=-\text{J}_1$ gives that 
$$\text{det}\left(
\begin{array}{cc}
 \text{J}_0(k \sqrt{n} ) & \, \,  -\text{J}_0(k ) \\
- k\sqrt{n}  \text{J}_1(k \sqrt{n} )  & \, \,  k \text{J}_1(k )- \eta \text{J}_0(k )
\end{array}
\right)=0$$
 where $\text{J}_1$ is the Bessel function of order 1. Therefore, we have that $k$ is an ITE provided that $d_\eta(k)=0$ where\\
$$ d_\eta(k)=k \text{J}_0(k \sqrt{n})\text{J}_1(k)-k \sqrt{n}\text{J}_0(k)\text{J}_1(k \sqrt{n}) -\eta \text{J}_0(k\sqrt{n} )\text{J}_0(k).$$
Just as the previous case we have that there existence of infinitely many complex roots of $d_0(k)$.

\begin{theorem}\label{complex-te-2}
Assume that $n$ and $\eta$ are positive constants. Then there exists infinitely many complex-valued ITEs for the disk in $\R^2$, provided that $\eta$ is sufficiently small. 
\end{theorem}

Again, we obtain two real and three complex-valued ITEs 
in the domain $\Omega=[0,10]\times [0,10]\mathrm{i}\subset \mathbb{C}$.
\begin{table}[!ht]
\centering
\begin{tabular}{c|c|c|c|c}
 4.159236    & 8.261173  & $2.363421+0.781661\mathrm{i}$&  $5.646922+0.735262\mathrm{i}$  &  $ 8.814961+0.318519\mathrm{i}$  
\end{tabular}
\caption{The eigenvalues for the unit circle for $n=3$ and $\eta= 0$.}
\end{table}

% The ITEs are $4.159\,236$, $8.261\,173$, $2.363\,421+0.781\,661\mathrm{i}$, $5.646\,922+0.735\,262\mathrm{i}$, and 
%$8.814\,961+0.318\,519\mathrm{i}$. 
As we can see in Table \ref{tableConvergenceCircle} the estimated order of convergence for the five ITEs is linear.
\begin{table}[!ht]
\centering
\begin{tabular}{c|ccccc}
      $\eta$    & $\mathrm{EOC}^{(1)}$ & $\mathrm{EOC}^{(2)}$ & $\mathrm{EOC}^{(3)}$ & $\mathrm{EOC}^{(4)}$ & $\mathrm{EOC}^{(5)}$\\
 \hline
 1    &     &     &     &     &\\
 1/2  &1.018&1.108&1.081&0.997&0.950\\
 1/4  &1.014&1.055&1.056&1.002&0.977\\
 1/8  &1.009&1.028&1.030&1.002&0.990\\
 1/16 &1.005&1.014&1.015&1.001&0.996\\
 1/32 &1.002&1.007&1.008&1.001&0.997\\
 1/64 &1.001&1.003&1.004&1.000&1.007\\
 1/128&1.001&1.002&1.002&0.999&0.989\\
 1/256&1.000&1.001&1.001&1.000&0.999\\
 \hline
\end{tabular}
\caption{\label{tableConvergenceCircle}The estimated order of convergence for five ITEs for a unit circle using $n=3$ as $\eta\rightarrow 0$.}
\end{table}

For the unit circle, we define the squared $L^2$ error by
\[
\epsilon_\eta^{(w)}=2\pi\int_0^1 \left|w_\eta(r)-w_0(r)\right|^2r\,\mathrm{d}r 
\]
with the obvious definition of $\epsilon_\eta^{(v)}$ and the appropriate use of the Bessel function $\mathrm{J}_0$ instead of $j_0$. In Table 
\ref{tableConvergenceCircleEigenfunctions} we show the estimated order of convergence for the first real and the first complex-valued ITEs. We again observe a linear convergence rate.

\begin{table}[!ht]
\centering
\begin{tabular}{c|cc|cc}
      $\eta$    & $\mathrm{EOC}^{(w)}$ & $\mathrm{EOC}^{(v)}$ & $\mathrm{EOC}^{(w)}$ & $\mathrm{EOC}^{(v)}$ \\
 \hline
 1    &     &     &     &     \\
 1/2  &1.034&0.941&1.082&1.071\\
 1/4  &1.022&0.967&1.083&1.078\\
 1/8  &1.012&0.983&1.047&1.046\\
 1/16 &1.006&0.992&1.025&1.024\\
 1/32 &1.003&0.996&1.013&1.012\\
 1/64 &1.002&0.998&1.006&1.006\\
 1/128&1.001&0.999&1.003&1.003\\
 1/256&1.000&0.999&1.002&1.002\\
 \hline
\end{tabular}
\caption{\label{tableConvergenceCircleEigenfunctions}The estimated order of convergence for two pairs of ITEs for a unit circle 
using $n=3$ as $\eta\rightarrow 0$.}
\end{table}

%%%%%%%%%%%%%%%%%%%%%%%%%%%%%%%%%%%%%%%%%%%%%%%%%%
\subsection{The inside-outside-duality method}
Recently, a new method has been invented to determine ITEs (see \cite{armin}) which has been successfully applied to various scattering problems (see \cite{lechrennoch, peters, peterskleefeld, peterslech}). It is another approach which does not need to solve a non-linear eigenvalue problem as done in \cite{kleefeldITP, kleefeldhabil}.

In this section, we show numerically that we are able to determine the ITEs using the inside-outside-duality approach,
although the theory still has to be carried out. 
We refer the reader to \cite{peterskleefeld} for the details of the implementation of the inside-outside-duality method and \cite{kleefeldE} for some recent advances. This approach connects the eigenvalues of the far field operator to the ITEs. Let $\lambda_j(k) \in \C$ for $j \in \N$ be the eigenvalues of ${F}={F}_k$  defined in \eqref{FF-operator} and $\phi_j(k)=\text{arg} \big\{ \lambda_j(k) \big\}$. In \cite{peterskleefeld} it is shown that for $\eta =0$ the function $\phi(k) =\max_{j \in \N} \phi_j(k)$ satisfies 
$$\phi(k) \rightarrow \pi \quad \text{ as } \quad k \rightarrow k_0 \, \, \, \text{ a transmission eigenvalue.}$$
We choose an equidistant grid $[1,5]$ with grid size $0.01$ and plot the phase curves $p$ against the wave number $k$. As we can see in Figure \ref{iod} we are able to determine five ITEs for the cases $\eta=0.1$, $0.5$, $1$, and $3$.

\begin{table}[!ht]
\centering
\begin{tabular}{c|ccccc}
      $\eta$    & 1. & 2. & 3. & 4. & 5.\\
 \hline
 0.1    &  3.10   &   3.13  &    3.68  &  4.25   & 4.82 \\
 0.5    & 2.97  & 3.08  & 3.64   &  4.22 & 4.79\\
 1       &2.79 &  3.02  & 3.60  &4.18  & 4.76\\
 3       &2.20   &2.80   &3.43   &4.04   &4.64\\
 \hline
\end{tabular}
\caption{ The five reconstructed ITEs for the unit sphere for $n=4$ by the inside-outside-duality method}
\end{table}

Comparing the results with \cite[Table 1]{te-cbc} shows that we are able to obtain the ITEs within the accuracy of the chosen grid size. Interestingly, the value $3.141\,593$ as reported in \cite[Table 1]{te-cbc} is missing.
There is another important observation to be mentioned. It has been noted that the phase curves are monotonically increasing towards $\pi$ for the ITEs (with $\eta=0$) for various obstacles (see \cite{peterskleefeld}), but this is not the case for $\eta \neq 0$ as one can clearly see in the second row of Figure \ref{iod}.

\begin{figure}[!ht]
\centering
 \includegraphics[height=5.5cm]{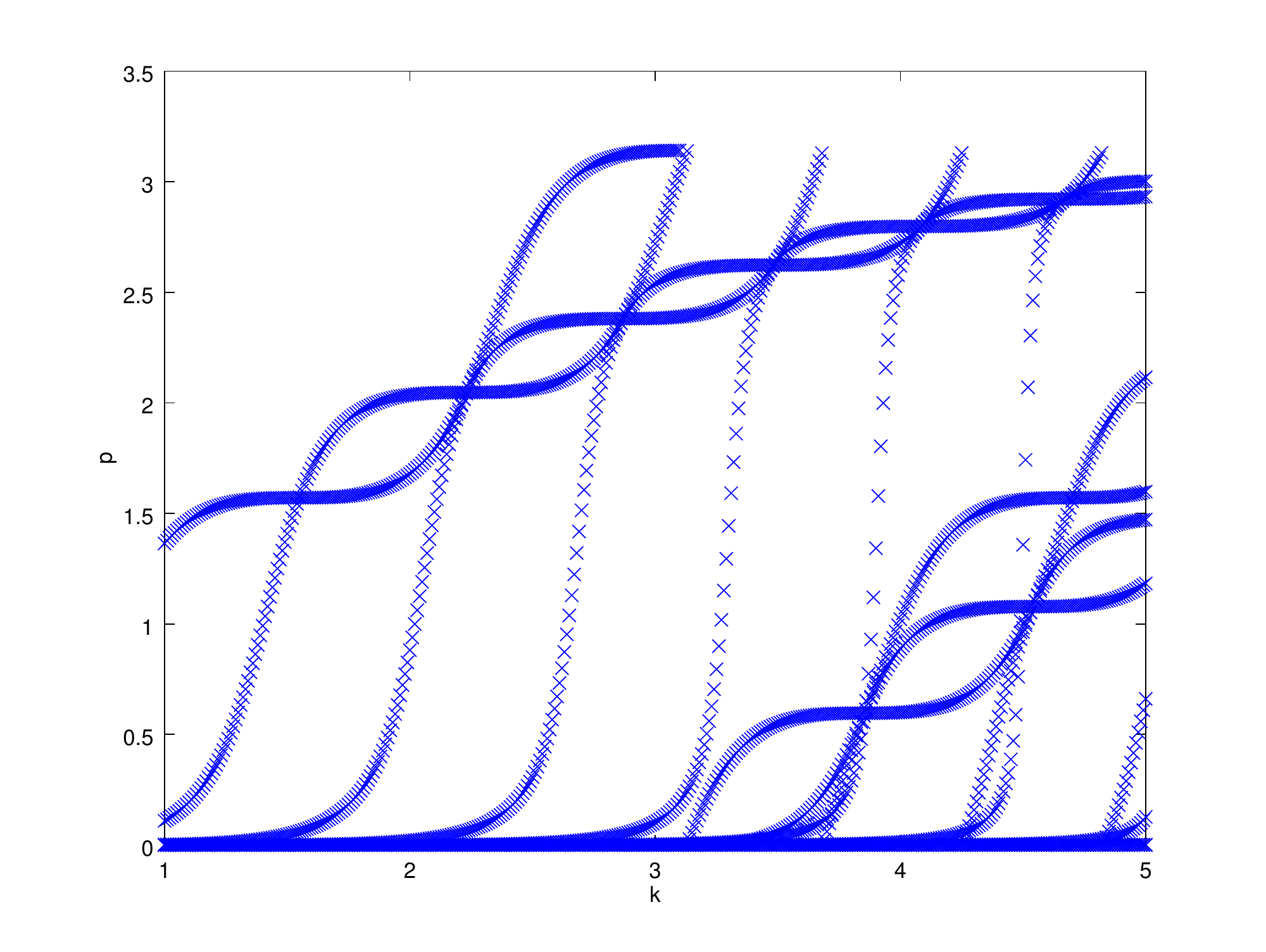}
 \includegraphics[height=5.5cm]{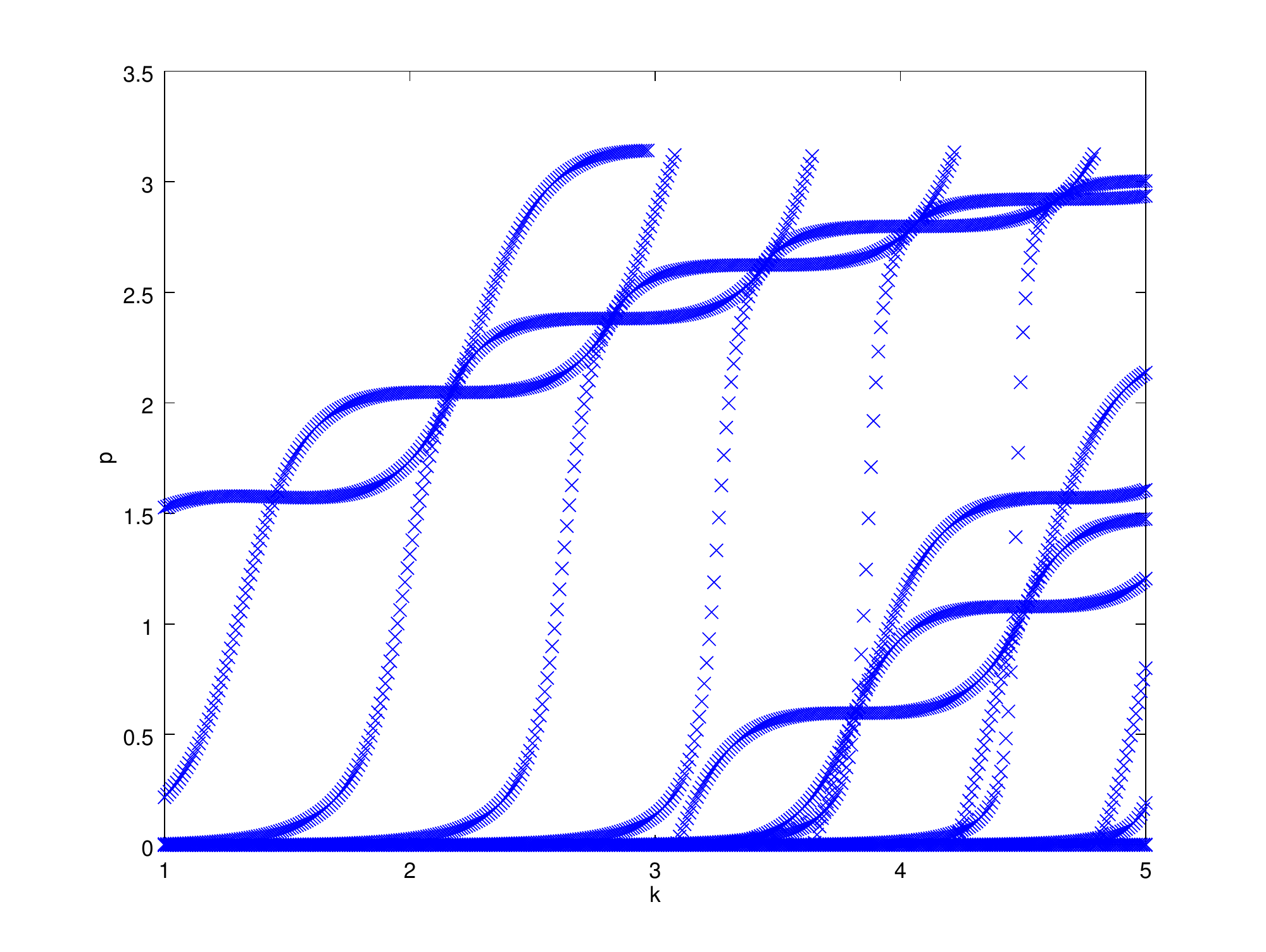}
  \includegraphics[height=5.5cm]{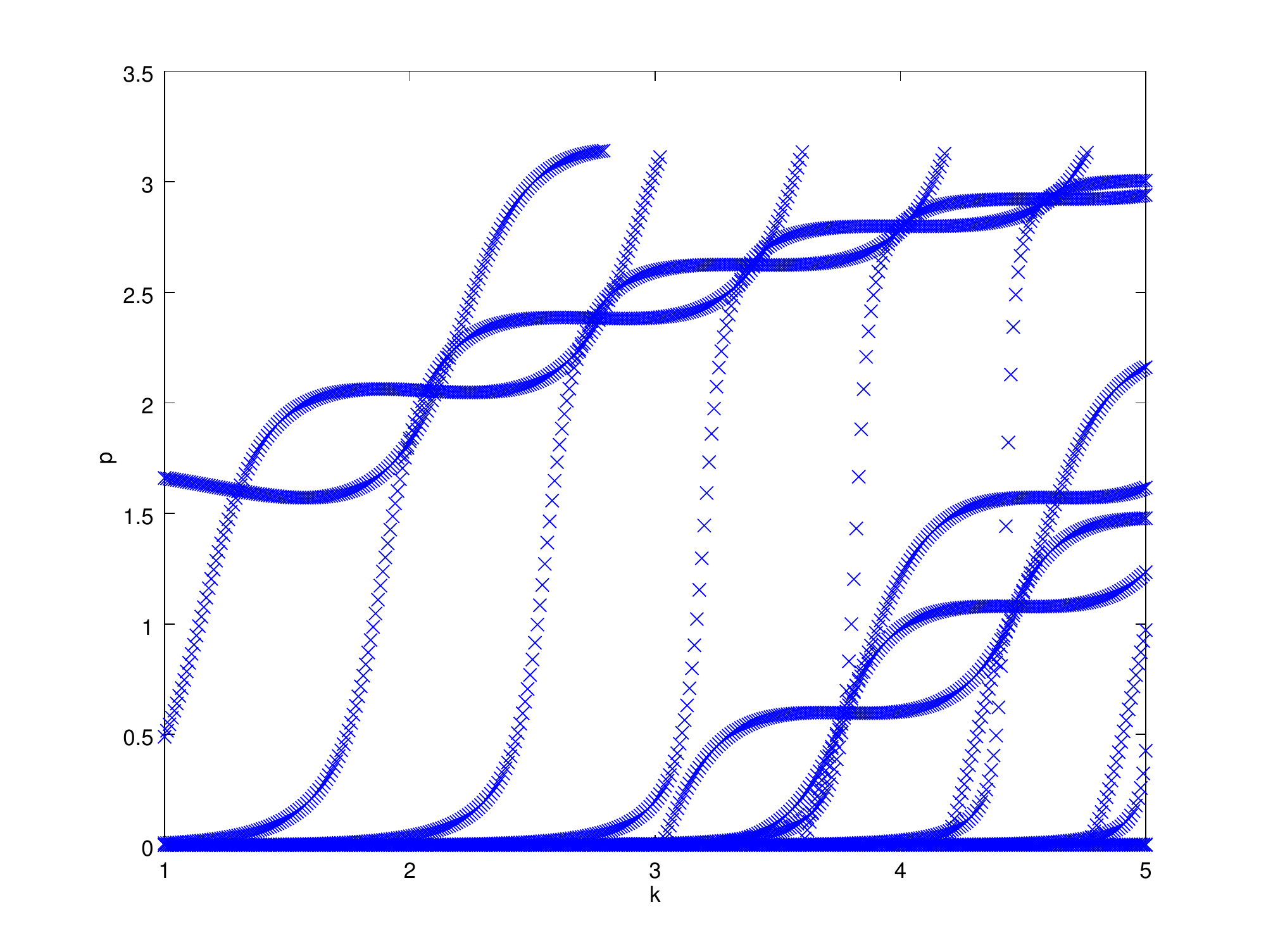}
  \includegraphics[height=5.5cm]{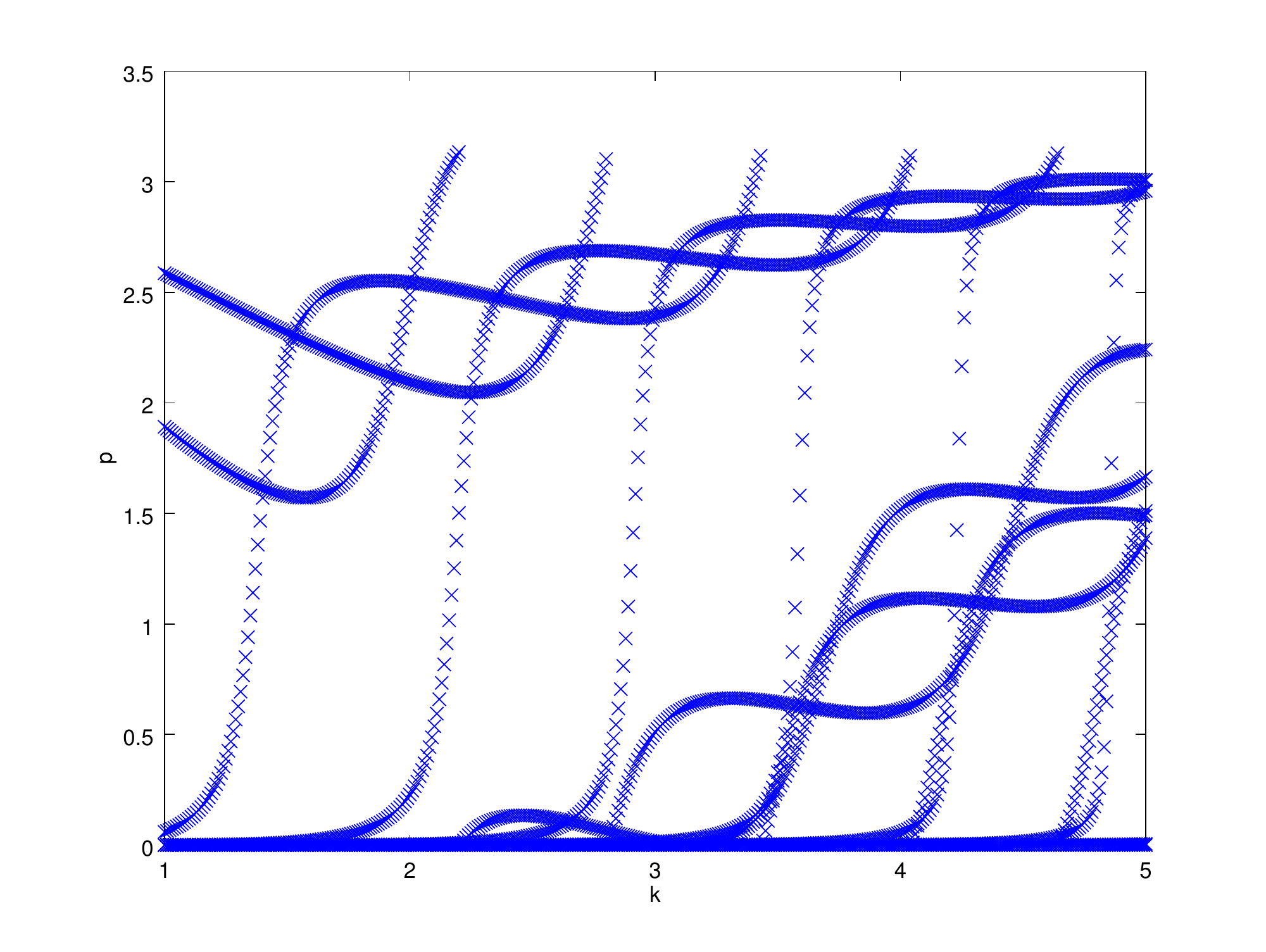}
 \caption{\label{iod}First row: The phase curves for a unit sphere using $\eta=0.1$ and $\eta=0.5$. Second row: The phase curves for a unit sphere using $\eta=1$ and $\eta=3$.}
\end{figure}

Additionally, we show that we are also able to find the ITEs with the inside-outside-duality method for a unit circle using the same parameters as before. In Figure \ref{iod2} we show the phase curves for the two cases $\eta=1$ and $\eta=3$. Again, it is interesting to observe that the phase curves are several times increasing and decreasing before approaching the value $\pi$.

\begin{table}[ht!]
\centering
\begin{tabular}{c|cccccc}
      $\eta$    & 1. & 2. & 3. & 4. & 5. & 6.\\
 \hline
 1       &2.77 &  3.29  & 3.31  &3.89  & 4.47 &  $-$ \\
 3       &2.49   &3.12   &3.14   &3.74   &4.34 &4.94 \\
 \hline
\end{tabular}
\caption{The reconstructed ITEs for the unit circle for $n=4$ by the inside-outside-duality method}
\end{table}

\begin{figure}[ht!]
\centering
  \includegraphics[height=5.5cm]{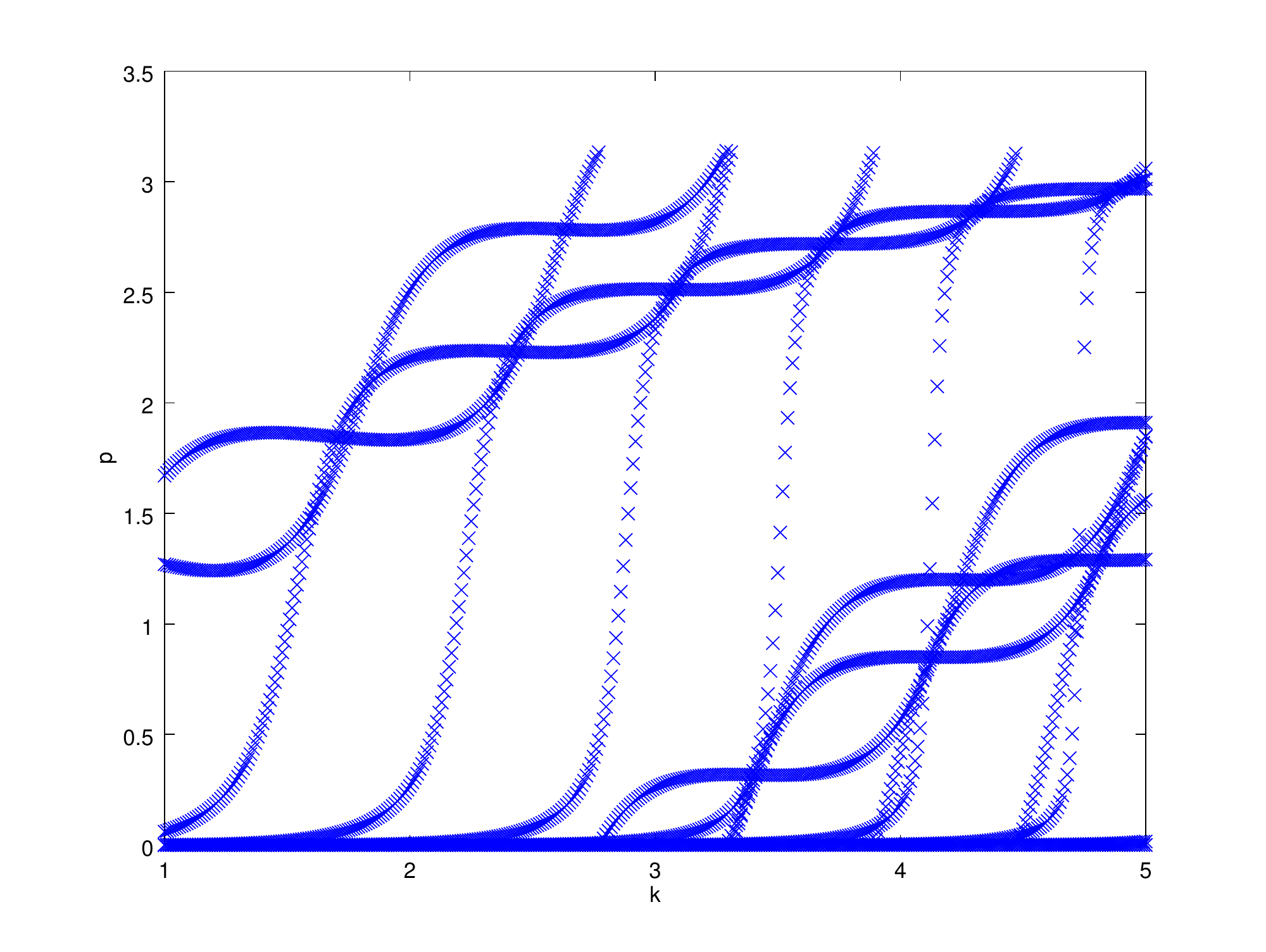}
  \includegraphics[height=5.5cm]{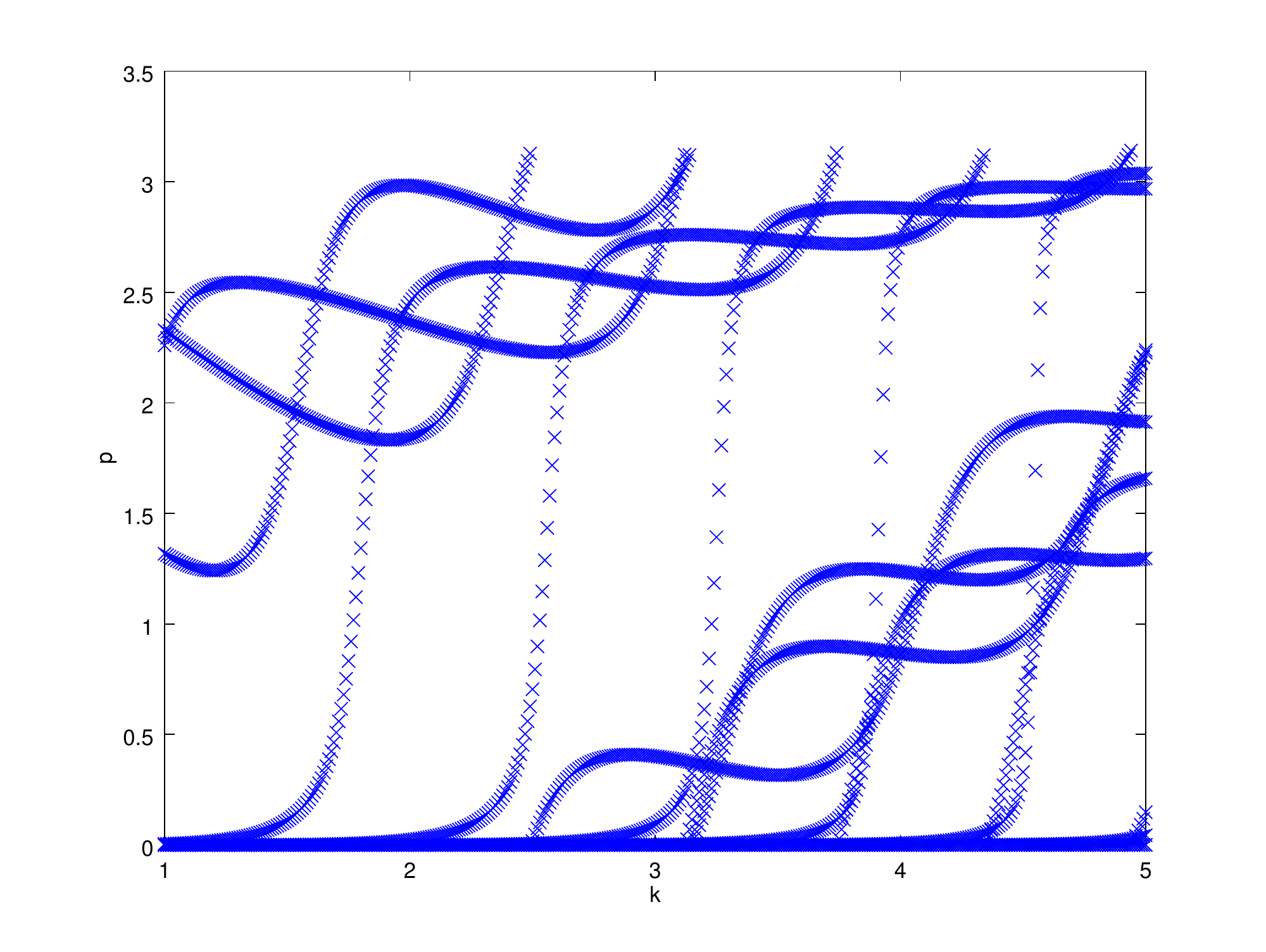}
 \caption{\label{iod2}The phase curves for a unit circle using $\eta=1$ and $\eta=3$.}
\end{figure}

%%%%%%%%%%%%%%%%%%%%%%%%%%%%%%%%%%%%%%%%%%%%
\section{Summary and conclusion}
In this article, it has been shown how to obtain the conductivity parameter $\eta$ from either far field or near field data with the main focus on the last case. Further, uniqueness is proven and an algorithm is presented to retrieve the parameter $\eta$ where a constant index of refraction is assumed. Next, a detailed study of the conductive eigenvalue problem is performed. Precisely, it has been illustrated how to acquire conductive interior transmission eigenvalues from either far field or near field data for the case of a non-absorbing media. Additionally, it has been proven that the conductive interior transmission eigenvalues and corresponding eigenfunctions converge linearly to the classic interior transmission eigenvalues and eigenfunctions, respectively. Further, existence and discreteness of conductive interior transmission eigenvalues is shown for an absorbing media. Finally, numerical experiments confirm the theoretical findings for two- and three-dimensional scatters. Above all, it is illustrated that one is able to get conductive interior transmission eigenvalue using the inside-outside duality method. However, the theory has not been carried out and is therefore a subject of future research. Some questions that are still open for this problem are:
\begin{enumerate}
\item Asymptotic expansion of the eigenvalues and eigenfunctions as $\eta \rightarrow 0$.
\item Analysis of the interior conductive eigenvalue problem as $\eta \rightarrow \infty$. 
\item Inverse spectral problem of reconstructing either $n$ and/or $\eta$ from a knowledge of the spectral data for a spherically stratified media.
\item Existence of complex-valued interior conductive eigenvalues for more general domains and coefficients.
\end{enumerate}

%%%%%%%%%%%%%%%%%%%%%%%%%%%%%%%%%%%%%%%%%%%%%%%%%%%%%%%%%%%%
\renewcommand{\refname}{REFERENCES}

\end{document}